\documentclass[11pt,letterpaper,reqno]{amsart}
\usepackage{fullpage}
\usepackage[top=2cm, bottom=4.5cm, left=2.5cm, right=2.5cm]{geometry}
\usepackage{amsmath,amsthm,amsfonts,amssymb,amscd,systeme}
\usepackage{geometry,tikz,tikz-cd}
\usepackage{empheq}
\usepackage{lipsum}
\usepackage{lastpage}
\usepackage{thmtools}
\usepackage{enumerate}
\usepackage[shortlabels]{enumitem}
\usepackage{fancyhdr}
\usepackage{mathrsfs}
\usepackage{xcolor}
\usepackage{graphicx}
\usepackage{mathtools}
\usepackage{listings}
\usepackage{bbm}

\usepackage{hyperref}
\usepackage[makeroom]{cancel}
\usepackage[utf8]{inputenc}

\usepackage[all]{xy}
\xyoption{matrix}
\xyoption{arrow}

\usetikzlibrary{arrows,decorations.pathmorphing,backgrounds,positioning,fit,petri}

\tikzset{help lines/.style={step=#1cm,very thin, color=gray},
help lines/.default=.5} 
\tikzset{thick grid/.style={step=#1cm,thick, color=gray},
thick grid/.default=1} 

\DeclareMathAlphabet\mathbfcal{OMS}{cmsy}{b}{n}

\makeatletter
\NewDocumentCommand{\sump}{e{_}}
 {%
  \DOTSB
  \mathop{\IfNoValueTF{#1}{\sump@{}}{\sump@{#1}}}%
  \nolimits
 }
\newcommand{\sump@}[1]{\mathpalette\sump@@{#1}}
\newcommand{\sump@@}[2]{%
  \ifx#1\displaystyle
    {\sump@display{#2}}%
  \else
    \sum@\nolimits'_{#2}%
  \fi
}
\newcommand{\sump@display}[1]{%
  \sbox\z@{$\m@th\displaystyle\sum@\nolimits'$}%
  \sbox\tw@{$\m@th\displaystyle\sum@\limits_{#1}$}%
  \sbox\@tempboxa{$\m@th\displaystyle'$}
  \mathop{\sum@\nolimits' \kern-\wd\@tempboxa}\limits_{#1}%
  \ifdim\wd\z@>\wd\tw@
    \kern\dimexpr\wd\z@-\wd\tw@\relax
  \fi
}
\makeatother

\mathtoolsset{showonlyrefs=true}
\numberwithin{figure}{section}
\numberwithin{table}{section}

\theoremstyle{definition}

\theoremstyle{plain}
\newcommand{\thistheoremname}{}
\newtheorem*{genericthm*}{\thistheoremname}
\newenvironment{namedthm*}[1]
  {\renewcommand{\thistheoremname}{#1}%
   \begin{genericthm*}}
  {\end{genericthm*}}

\hypersetup{%
  colorlinks=true,
  linkcolor=blue,
  citecolor=blue,
  linkbordercolor={0 0 1}
}

\allowdisplaybreaks

 \newtheorem{theorem}{Theorem}[section]
 \newtheorem{corollary}{Corollary}[section]
 \newtheorem{lemma}[theorem]{Lemma}
 \newtheorem{proposition}[theorem]{Proposition}
 \theoremstyle{definition}
 \newtheorem{definition}{Definition}[section]
 \theoremstyle{definition}
 
 \theoremstyle{remark}
 \newtheorem{remark}{Remark}[section]
\newtheorem{conjecture}[theorem]{\bf Conjecture}

 \numberwithin{equation}{section}

 \theoremstyle{remark}
 \newtheorem*{Notations*}{Notations}
 \newtheorem*{Proof*}{Proof}

\usepackage{latexsym}
\usepackage{amssymb}
\usepackage{nicematrix}

\newcommand{\ben}{\begin{equation}}
\newcommand{\een}{\end{equation}}

\DeclareMathOperator{\Sl}{SL}
\DeclareMathOperator{\Gl}{GL}
\DeclareMathOperator{\Span}{Span}
\DeclareMathOperator{\Hom}{Hom}

\DeclareMathOperator{\re}{Re}

\DeclareMathOperator{\sign}{sign}

\DeclareMathOperator{\Tr}{Tr}

\setlist[enumerate]{leftmargin=*,widest=0}

\makeatletter
\def\subsection{\@startsection{subsection}{2}%
  \z@{.5\linespacing\@plus.7\linespacing}{.3\linespacing}%
  {\normalfont\bfseries}}

\def\subsubsection{\@startsection{subsubsection}{3}%
  \z@{.5\linespacing\@plus.7\linespacing}{.3\linespacing}%
  {\normalfont\bfseries}}
\makeatother

\title{Twisted periods of modular forms}

\begin{document}
\author{Tianyu Ni}
\address{School of Mathematical and Statistical Sciences\\
Clemson University\\
Clemson, SC 29634-0975\\
USA}
\email{tianyuni1994math@gmail.com}
\author{Hui Xue}
\address{School of Mathematical and Statistical Sciences\\
Clemson University\\
Clemson, SC 29634-0975\\
USA}
\email{huixue@clemson.edu}
\begin{abstract}
    Let $S_k$ denote the space of cusp forms of weight $k$ and level one. For $0\leq t\leq k-2$ and primitive Dirichlet character $\chi$ mod $D$, we introduce twisted periods $r_{t,\chi}$ on $S_k$. We show that for a fixed natural number $n$, if $k$ is sufficiently large relative to $n$ and $D$, then any $n$ periods with the same twist but different indices are linearly independent. We also prove that if $k$ is sufficiently large relative to $D$ 
    then any $n$ periods with the same index but different twists mod $D$ are linearly independent. These results are achieved by studying the trace of the products and Rankin-Cohen brackets of Eisenstein series of level $D$ with nebentypus. 
    Moreover, we give two applications of our method. First, we prove certain identities that evaluate convolution sums of twisted divisor functions. Second, we show that Maeda's conjecture implies a non-vanishing result on twisted central $L$-values of normalized Hecke eigenforms. 
\end{abstract}
\maketitle

\section{Introduction}
Let $k\geq2$ be an integer. For $N\geq1$ and a Dirichlet character $\chi$ mod $N$, let $M_k(N,\chi)$ and $S_k(N,\chi)$ be the space of modular forms and cusp forms of weight $k$, level $N$ and nebentypus $\chi$, respectively. We simply write $M_k(N)$ and $S_k(N)$ if $\chi$ is trivial, and use $M_k$ and $S_k$ to denote $M_k(1)$ and $S_k(1)$. 
For each $0\leq t\leq k-2$, the $t$-th period of $f\in S_k$ is given by
\begin{align}
    r_t(f):=\int_0^{i\infty}f(z)z^tdz=\frac{t!}{(-2\pi i)^{t+1}}L(f,t+1).
\end{align}
Here the $L$-function of $f(z)=\sum_{n\geq1}a_f(n)q^n$ is $L(f,s)=\sum_{n\geq1}a_f(n)n^{-s}$, which converges for $\re(s)$ sufficiently large and can be extended analytically to the whole complex plane. Note that each $r_t$ can be regarded an element in $S_k^{\ast}:=\Hom_{\mathbb{C}}(S_k,\mathbb{C})$. The result of Eichler and Shimura (\cite{Eichlerperiod,Manin1973,Kohnen1984}) asserts that odd periods $r_1,r_3,...,r_{k-3}$ (or even periods $r_0,r_2,...,r_{k-2}$) span the vector space $S_k^{\ast}$. However, these periods are not linearly independent. In fact, they are subject to many linear dependence relations, called the Eichler-Shiumra relations \cite{Manin1973}.
     This leads to a natural question. Which periods are linearly independent? The first work in this direction that we are aware of is \cite{Fukuhara07}, in which Fukuhara found an explicit subset of odd periods that forms a basis for $S_k^{\ast}$. Recently,  Lei et al. \cite{oddperiods,evenperiods} have provided some evidence for the linear independence of odd periods and even periods for modular forms, respectively. Furthermore, the linear independence between an odd and an even periods has been addressed in \cite{Xueoddandevenperiods}.

In this paper, we consider the \textit{twisted periods} of modular forms, which are natural generalizations of periods. Let $\chi$ be a primitive Dirichlet character mod $D$. For $0\leq n\leq k-2$, we define the $n$-th twisted period for $f\in S_k$ (\cite[p.~978]{twistedLvaluesFukuhara}) to be 
    \begin{align}
        r_{n,\chi}(f):=\int_0^{i\infty} f_{\chi}(x) z^n dz=\frac{n!}{(-2\pi i)^{n+1}}L(f,\chi,n+1).\label{eq:periodandLvalue}
    \end{align}
Here the twist $f_{\chi}(z):=\sum_{n\geq1}\chi(n)a_f(n)q^n$ is an element in $S_{k}(D^2,\chi^2)$ (\cite[Proposition 14.19]{Iwaniecbook}).
The twisted $L$-function $L(f,\chi,s):=\sum_{n\geq1}\chi(n)a_f(n)n^{-s}$, originally defined for $\re(s)\gg0$, can be analytically continued to the whole complex plane. For a normalized Hecke eigenform $f\in S_k$, we have 
(\cite[Proposition 14.20]{Iwaniecbook}): 
    \begin{align}
 \Lambda(f,\chi,s)=i^k w_{\chi} \Lambda(f,\overline{\chi},k-s),   \label{eq:functionaleqoftwistedL}
 \end{align}
where $\Lambda(f,\chi,s)=(2\pi)^{-s}\Gamma(s)L(f,\chi,s)$ is the completed twisted $L$-function and $ w_{\chi}:=\frac{G(\chi)^2}{D}$. 
One can view each $r_{n,\chi}$  as an element in $S_k^{\ast}$.

Since twisted periods can be regarded as elements in $S_{k}^{\ast}$, the first question about these periods is whether they span the whole space of $S_k^{\ast}$. The case of $D=4$ has been treated by Kina \cite[Corollary 5.8]{KinadoubleEisenstein2024}. The second question is whether we can find Eichler-Shimura type linear relations  among them. 
Furthermore, we would like to know which periods are linearly independent. The first two questions are beyond the scope of this paper. But, we will propose related conjectures in the last section. In this paper, extending the ideas of \cite{oddperiods, evenperiods}, we will provide some evidence for the linear independence of twisted periods. 

We now state our results. First, we consider the linear independence of twisted periods for the same character.

\begin{theorem}\label{thm:mainthmfixDdifferentell}
    Let $n\geq1$ be an integer, $D\geq1$ an odd square-free integer, and $\chi$ be a primitive Dirichlet character mod $D$. For $K\gg_{n,D}1$, if $3\leq\ell_1<\ell_2<...<\ell_n\leq\frac{K-2}{2}$ are integers such that $\chi(-1)=(-1)^{\ell_i}$ for all $1\leq i\leq n$, then the set of twisted periods $\{r_{\ell_i-1,\chi}\}_{i=1}^n$ on $S_{K}$ is linearly independent. 
\end{theorem}
\begin{theorem}\label{thm:mainthmfixDdifferentell2}
    Let $n\geq1$ be an integer, $D\geq1$ an odd square-free integer, and $\chi$ be a primitive Dirichlet character mod $D$. For $K\gg_{n,D}1$, if $3\leq\ell_1<\ell_2<...<\ell_n\leq\frac{K-4}{2}$ are integers such that $\chi(-1)=(-1)^{\ell_i}$ for all $1\leq i\leq n$, then the set of twisted periods $\{r_{\ell_i,\chi}\}_{i=1}^n$ on $S_{K}$ is linearly independent. 
\end{theorem}

We then consider the linear independence of twisted periods with the same index but for different characters mod $D$. 
\begin{theorem}\label{thm:twistbydifferentD}
    Let $D\geq1$ be an odd square-free integer. For $K\gg_D1$, if $3\leq\ell\leq\frac{K-2}{2}$ is an integer, $\chi_1,...,\chi_n$ are primitive Dirichlet characters mod $D$ such that $\chi_i(-1)=(-1)^{\ell}$ and $\chi_i(2)$ are distinct for all $1\leq i\leq n$, then the set of twisted periods $\{r_{\ell-1,\chi_i}\}_{i=1}^n$ on $S_K$ is linearly independent.
\end{theorem}
\begin{theorem}\label{thm:twistbydifferentD2}
     Let $D\geq1$ be an odd square-free integer. For $K\gg_D1$, if $3\leq\ell\leq\frac{K-4}{2}$ is an integer and $\chi_1,...,\chi_n$ are primitive Dirichlet characters mod $D$ such that $\chi_i(-1)=(-1)^{\ell}$ and $\chi_2(2)$ are distinct for all $1\leq i\leq n$, then the set of twisted periods $\{r_{\ell,\chi_i}\}_{i=1}^n$ on $S_K$ is linearly independent.
\end{theorem}
\begin{remark}
    The assumption that $\chi_i$'s take distinct values at $2$ in Theorems \ref{thm:twistbydifferentD} and \ref{thm:twistbydifferentD2} can be replaced by $\chi_i$'s taking distinct values at an arbitrary fixed prime $p$, in which case our results hold for $K\gg_{p,D}1$.
\end{remark}
\begin{remark}We want to point out that when $D=1$, the twisted periods involved in Theorems \ref{thm:mainthmfixDdifferentell} and \ref{thm:twistbydifferentD} are the odd periods considered in \cite{oddperiods}, and the twisted periods treated in Theorems \ref{thm:mainthmfixDdifferentell2} and \ref{thm:twistbydifferentD2} are exactly the even periods studied in \cite{evenperiods}.\end{remark}
Throughout the paper we assume that $D\geq1$ is a positive odd square-free integer in order to avoid technical complications.

We sketch the rough idea of the proof of Theorems \ref{thm:mainthmfixDdifferentell} and \ref{thm:twistbydifferentD}. The key is to  study the linear independence of the kernel functions  that give rise to the twisted periods. For $f,g\in S_K$, let $\langle f,g\rangle$ denote the Petersson inner product. Then there is a cusp form $R_{K,n,\chi}$ such that $r_{n,\chi}(f)=\langle R_{K,n,\chi},f\rangle$ for all $f\in S_K$. In general, $R_{K,n,\chi}$ is expected to have transcendental Fourier coefficient, which is not convenient to study. Instead of considering $R_{K,n,\chi}$, we study the modular form $\Tr_1^D(G_{\ell,\chi}(z)G_{K-\ell,\overline{\chi}}(z))\in M_K$. Here the Eisenstein series $G_{k,\chi}$ of weight $k$ (see Section \ref{sect:rankin-selberg} for details) is given by 
\begin{align}
    G_{k,\chi}(z):&=\sum_{n=0}^{\infty}\sigma_{k-1,\chi}(n)q^n\in M_{k}\left(D,\chi\right) \label{eq:Fourier-expan-G}\\
    \sigma_{k-1,\chi}(n):&=\begin{cases}
        \frac{(k-1)!D^k}{(-2\pi i)^kG(\overline{\chi})}L(k,\overline{\chi}) &n=0,\\
         \sum\limits_{d\mid n}\chi(d)d^{k-1}& n\geq1.
    \end{cases}\quad\label{eq:twistsumdivisor}
\end{align}
Additionally, for $N\mid M$, $\Tr_N^M$  is the trace map 
\begin{align} \label{eq:tracemap}
    \Tr_N^M: M_{m}(M)\rightarrow M_m(N),\quad g\mapsto \sum_{\gamma\in\Gamma_0(M)\backslash\Gamma_0(N)}g|_m\gamma,
\end{align}
where for any real number $m$ and $\gamma=\begin{bsmallmatrix}
    a&b\\c&d
\end{bsmallmatrix}\in \Gl_2^+(\mathbb{R})$  the slash operator \cite[Theorem 7.1]{Cohen'smodularformC_k} is
\begin{align} \label{eq:slashoperator}
g(z)|_m\gamma=\det(\gamma)^{m/2} (cz+d)^{-m} g\left(\frac{az+b}{cz+d}\right).
\end{align}
We will show that the cuspidal projection of $\Tr_1^D(G_{\ell,\chi}(z)G_{K-\ell,\overline{\chi}}(z))$ is \begin{align}
    \mathcal{F}_{K,\ell,\chi}(z):=\Tr_1^D(G_{\ell,\chi}(z)G_{K-\ell,\overline{\chi}}(z))-\frac{2\sigma_{\ell-1,\chi}(0)\sigma_{k-1,\overline{\chi}}(0)}{\zeta(1-K)}G_K(z),
\label{eq:F-def}\end{align}
where  $G_K(z)=\frac{\zeta(1-K)}{2}+\sum_{n\geq1}\sigma_{K-1}(n)q^n$, $\zeta(1-K)=-\frac{B_K}{K}$, and $B_K$ is the $K$-th Bernoulli number, see Section \ref{sect:Fourier}. Then we can relate the twisted period $r_{\ell-1,\chi}$ and the cusp form $\mathcal{F}_{K,\ell,\chi}$ through the Rankin-Selberg method (Proposition \ref{prop:Rankin-selbergperiod}), which in turn allows us to show the equivalence between the linear independence of $\{r_{\ell_i-1,\chi}\}_{i=1}^n$ and linear independence of $\{\mathcal{F}_{K,\ell_i,\chi}\}_{i=1}^n$ (Proposition \ref{prop:linearindeperiodtomodularform}). Applying the same methodology in \cite{oddperiods,evenperiods}, we study the linear independence of $\{\mathcal{F}_{K,\ell_i,\chi}\}_{i=1}^n$ by investigating the matrix $M_{K,\ell_1,...,\ell_n,\chi}$ \eqref{eq:matrix1} formed by the Fourier coefficients of $\{\mathcal{F}_{K,\ell_i,\chi}\}_{i=1}^n$. Similarly, we can show that $\{r_{\ell-1,\chi_i}\}_{i=1}^n$ is linearly independent if and only if $\{\mathcal{F}_{K,\ell,\chi_i}\}_{i=1}^n$  is linearly independent. Theorem \ref{thm:twistbydifferentD} can be proved by showing that the matrix  $P_{K,\ell, \chi_1,...,\ell_n}$ \eqref{eq:matrix3} formed by the Fourier coefficients of $\{\mathcal{F}_{K,\ell,\chi_i}\}_{i=1}^n$ is non-singular.

Theorems \ref{thm:mainthmfixDdifferentell2} and \ref{thm:twistbydifferentD2} can be proved in a similar fashion. But we need to consider the Rankin-Cohen bracket (see Section \ref{sect:rankin-selberg}) of Eisenstein series. Define 
\begin{align}\mathcal{G}_{K,\ell,\chi}(z)=\Tr^D_1([G_{\ell,\chi}(z),G_{K-2-\ell,\overline{\chi}}(z)]_1).\label{eq:G-def}
\end{align}
By the Rankin-Selberg convolution (Proposition \ref{prop:Rankin-selbergperiod2}), we can show the equivalence between the linear independence of $\{r_{\ell_i,\chi}\}_{i=1}^n$ ($\{r_{\ell,\chi_i}\}_{i=1}^n$) and the linear independence of $\{\mathcal{G}_{K,\ell_i,\chi}\}_{i=1}^n$ ($\{\mathcal{G}_{K,\ell,\chi_i}\}_{i=1}^n$). Then Theorems \ref{thm:mainthmfixDdifferentell2} and \ref{thm:twistbydifferentD2} can be proved by showing the non-singularity of  $N_{K,\ell_1,...,\ell_n,\chi}$ \eqref{eq:matrix2} and $Q_{K,\ell, \chi_1,...,\ell_n}$ \eqref{eq:matrix4}, respectively.

The paper is organized as follows. In Section \ref{sect:rankin-selberg}, we prove the equivalence between the linear independence of the twisted periods and the linear independence of the corresponding cusp forms using Rankin-Selberg convolutions. In Section \ref{sect:Fourier}, we first compute the Fourier coefficients of $\mathcal{F}_{K,\ell,\chi}$ and $\mathcal{G}_{K,\ell,\chi}$. Then we give the asymptotics of the Fourier coefficients of $\mathcal{F}_{K,\ell,\chi}$ and $\mathcal{G}_{K,\ell,\chi}$. In Section \ref{sect:coefmatrix}, we define the matrices $M_{K,\ell_1,...,\ell_n,\chi}, N_{K,\ell_1,...,\ell_n,\chi}, P_{K,\ell,\chi_1,...,\chi_n}$ and $Q_{K,\ell,\chi_1,...,\chi_n}$ formed by the Fourier coefficients of $\{\mathcal{F}_{K,\ell_i,\chi}\}_{i=1}^n, \{\mathcal{G}_{K,\ell_i,\chi}\}_{i=1}^n, \{\mathcal{F}_{K,\ell,\chi_i}\}_{i=1}^n$ and $\{\mathcal{G}_{K,\ell,\chi_i}\}_{i=1}^n$, respectively. We then give criteria for the non-singularity of these matrices. Finally, we prove Theorems \ref{thm:mainthmfixDdifferentell}, \ref{thm:mainthmfixDdifferentell2}, \ref{thm:twistbydifferentD} and \ref{thm:twistbydifferentD2} by showing the corresponding matrices are non-singular. In Section \ref{sect:applications}, we give two applications of our method. We prove some identities that evaluate convolution sums of twisted divisor functions. We also show that Maeda's conjecture implies a non-vanishing result on twisted central $L$-values of normalized Hecke eigenforms. In the last section, we propose several conjectures on spanning $S_K$ by twisted periods.

\section{Rankin-Selberg convolutions}\label{sect:rankin-selberg}
In this section, we prove the equivalence between the linear independence of twisted periods and the linear independence of the corresponding kernel cusp forms using Rankin-Selberg convolutions.  

First, we recall  some basic facts of Eisenstein series as developed in Miyake's book \cite[\S 7]{MiyakeMFbook}. Let $\chi$ and $\psi$ be Dirichlet characters mod $L$ and mod $M$, respectively. For any positive integer $k\geq3$, we put (\cite[\S7]{MiyakeMFbook})
\begin{align}
    E_{k}(z;\chi,\psi)=\sideset{}{'}\sum\limits_{m,n\in\mathbb{Z}}\chi(m)\psi(n)(mz+n)^{-k}.\nonumber
\end{align}
Here,  $\sideset{}{'}\sum$ is the summation over all pairs of integers $(m,n)$ except $(0,0)$. 
\label{lem:EisensteinFourierexpansion}
    \begin{lemma}[{\cite[Theorem 7.1.3]{MiyakeMFbook}}] 
Assume $k\geq3$. Let $\chi$ and $\psi$ be Dirichlet characters mod $L$ and mod $M$, respectively, satisfying $\chi(-1)\psi(-1)=(-1)^k$. Let $m_{\psi}$ be the conductor of $\psi$, and $\psi^0$ be the primitive character associated with $\psi$. Then 
    \begin{align}
        E_{k}(z;\chi,\psi)=C+A\sum_{n=1}^{\infty}a(n)e^{2\pi inz/M},\nonumber
    \end{align}
    where 
    \begin{align}
        A&=2(-2\pi i)^kG(\psi^0)/M^{k}(k-1)!,\nonumber\\
        C&=\begin{cases}2L_M(k,\psi) &\chi:{\rm the~principal~character},\nonumber\\
            0 &{\rm otherwise},
        \end{cases}\nonumber\\
        a(n)&=\sum_{0<c\mid n}\chi(n/c)c^{k-1}\sum_{0<d\mid(l,c)}d\mu(l/d)\psi^0(l/d)\overline{\psi^0}(c/d).\nonumber
    \end{align}
    Here $l=M/m_{\psi}$, $\mu$ is the M\"obius function, $L_M(k,\psi)=\sum_{n=1}^{\infty}\psi(n)n^{-k}$ is the Dirichlet series, and $G(\psi^0)$ is the Gauss sum of $\psi^0$.\end{lemma}
We need the following two types of Eisenstein series in this section. Let $\chi$ be a primitive Dirichlet character mod $D$ and $\mathbbm{1}$ be the principal character. 
Define
\begin{align}
    G_{k,\chi}(z):&=\frac{(k-1)!D^k}{2(-2\pi i)^kG(\overline{\chi})}E_k(Dz,\mathbbm{1},\overline{\chi})\nonumber\\&=\frac{(k-1)!D^k}{2(-2\pi i)^kG(\overline{\chi})}\sideset{}{'}\sum_{\substack{c,d\in\mathbb{Z}\\D\mid c}}\frac{\overline{\chi}(d)}{(cz+d)^k}.\label{eq:GkDforcomputingFourier}
\end{align}
From Lemma \ref{lem:EisensteinFourierexpansion} we know that
\begin{align*}
    G_{k,\chi}(z)&=\sum_{n=0}^{\infty}\sigma_{k-1,\chi}(n)q^n,
    \nonumber\\
    \sigma_{k-1,\chi}(n):&=\begin{cases}
        \frac{(k-1)!D^k}{(-2\pi i)^kG(\overline{\chi})}L(k,\overline{\chi}) &n=0, \nonumber\\
         \sum\limits_{d\mid n}\chi(d)d^{k-1}& n\geq1.
    \end{cases}
    \end{align*}
The Eisenstein series for \textit{the cusp at infinity} \cite[p.~272]{MiyakeMFbook} is defined as 
\begin{align*}
    E_{k,D}^{\ast}(z;\chi)=\sum_{\begin{bsmallmatrix}a&b\\c&d\end{bsmallmatrix}\in\Gamma_{\infty}\backslash\Gamma_0(D)}\frac{\chi(d)}{(cz+d)^k},
\end{align*}
where $\Gamma_{\infty}:=\{\pm\begin{bsmallmatrix}
    1 & n \\ 0 &1
\end{bsmallmatrix}:n\in\mathbb{Z}\}$. Note that \cite[(7.1.30)]{MiyakeMFbook}:
\begin{align}
    2L_{D}(k,\chi)E_{k,D}^{\ast}(z;\chi)=E_{k}(Dz;\mathbbm{1},\chi).\label{eq:midEisen}
\end{align}
The relation between $G_{k,\chi}(z)$ and $E_{k,D}^{\ast}(z;\overline{\chi})$ below is needed in the proof of Proposition \ref{prop:rankinselbergconvolution}.
\begin{lemma}
    Let $\chi$  be a primitive Dirichlet character mod $D$. Then 
    \begin{align}
    G_{k,\chi}(z)=\frac{(k-1)!D^kL(k,\overline{\chi})}{(-2\pi i)^kG(\overline{\chi})}E^{\ast}_{k,D}(z;\overline{\chi}).\label{eq:EisensteinforR-S}
\end{align}
\end{lemma}
\begin{proof}
    It is immediate from  \eqref{eq:GkDforcomputingFourier} and \eqref{eq:midEisen}.
\end{proof}
We then review the classical Rankin-Selberg method. For two elements $f$ and $g$ of $M_k(N)$ such that $fg$ is a cuspform, the Petersson inner product is given by 
\begin{align}
    \langle f,g\rangle_{N}=\int_{\Gamma_{0}(N)\backslash\mathbb{H}}f(z)\overline{g(z)}y^kd\mu,\label{eq:defofPeterssoninnerproduct}
\end{align}
Here $z=x+iy$ and $d\mu=dxdy/y^2$ is the $\Sl_2(\mathbb{R})$-invariant measure on the upper half plane $\mathbb{H}$. We use $\langle\cdot,\cdot\rangle$ to denote $\langle\cdot,\cdot\rangle_{N}$ if the level is clear from the context. 
\begin{definition}\label{def:rankincohen}
Let $f(z)\in M_{a}(\Gamma)$ and $g(z)\in M_b(\Gamma)$ be modular forms for some congruence subgroup $\Gamma$ of weights $a$ and $b$, respectively. For a nonnegative integer $e$, we define the $e$-th Rankin-Cohen bracket as 
\begin{align}
    [f(z),g(z)]_e := \sum_{r=0}^e (-1)^r\binom{e+a-1}{e-r}\binom{e+b-1}{r}f(z)^{(r)}g(z)^{(e-r)},\label{eq:defofrankin-cohenbracket}
\end{align}
where $f(z)^{(r)}$ is the $r$-th normalized derivative $f(z)^{(r)}:=\frac{1}{(2\pi i)^r}\frac{d^r f(z)}{dz^r}$ of $f$. Here $a,b$ can be in $\frac{1}{2}\mathbb{Z}$ and the binomial coefficients are defined through gamma functions. Moreover, $[f,g]_e\in M_{a+b+2e}(\Gamma)$ and $[f,g]_e\in S_{a+b+2e}(\Gamma)$ for $e>1$; see \cite[Theorem 7.1]{Cohen'smodularformC_k}. We remark that the Rankin-Cohen bracket defined in Zagier \cite[(73)]{Zagier1976} is related to \eqref{eq:defofrankin-cohenbracket} through $F_{e}^{(a,b)}(f(z),g(z))= (-2\pi i)^e e![f(z),g(z)]_e$; see  \cite[(1.1)]{evenperiods}.
\end{definition}

Recall the following classical result on Rankin-Selberg convolutions, which was reformulated and generalized in Zagier \cite{Zagier1976}, keeping in mind the difference between our definition of the Rankin-Cohen bracket and the one used therein.
\begin{lemma}[{\cite[Propsition 6]{Zagier1976}\label{lem:RankinSelbergZagier}}]
Let $k_1$ and $k_2$ be real numbers with $k_2\geq k_1+2>2$. Let $f(z)=\sum_{n=1}^{\infty}a(n)q^n$ and $g(z)=\sum_{n=0}^{\infty}b(n)q^n$ be modular forms in $S_k(N,\chi)$ and $M_{k_1}(N,\chi_1)$, where $k=k_1+k_2+2e, e\geq0$ and $\chi=\chi_1\chi_2$. 
Then 
    \begin{align}
        \langle f,[g,E^*_{k_2,N}(\cdot;\chi_2)]_e\rangle_{N}=\frac{(-1)^e}{e!}\frac{\Gamma(k-1)\Gamma(k_2+e)}{(4\pi)^{k-1}\Gamma(k_2)}\sum_{n=1}^{\infty}\frac{a(n)\overline{b(n)}}{n^{k_1+k_2+e-1}}.\label{eq:RankinSelberg}
    \end{align}
\end{lemma}

\begin{proposition}\label{prop:rankinselbergconvolution}
     Let $3\leq\ell\leq \frac{K-2}{2}$ be an integer, 
     $\chi$ a primitive Dirichlet character mod $D$ such that $\chi(-1)=(-1)^{\ell}$, and $f\in S_K$ be a normalized Hecke eigenform. Then 
     \begin{align}
         \langle f,G_{\ell,\chi}G_{K-\ell,\overline{\chi}}\rangle_{D}=\frac{\Gamma(K-1)(K-\ell-1)!D^{K-\ell}}{(4\pi)^{K-1}(2\pi i)^{K-\ell}\overline{G(\chi)}}L(f,K-1)L(f,\overline{\chi},K-\ell). \nonumber
     \end{align}
\end{proposition}
\begin{proof}
Let $k=K-\ell$. Note that \eqref{eq:EisensteinforR-S} implies that
\begin{align}
    G_{k,\overline{\chi}}(z)=\frac{(k-1)!D^kL(k,\chi)}{(-2\pi i)^{k}G(\chi)}E_{k,D}^{\ast}(z;\chi).
\end{align}
As $k\ge \ell+2$, using Lemma \ref{lem:RankinSelbergZagier} for $e=0$, 
\eqref{eq:Fourier-expan-G} and the fact that $\overline{\sigma_{\ell-1,\chi}(n)}=\sigma_{\ell-1,\overline{\chi}}(n)$, we get
\begin{align}
    \langle f,G_{\ell,\chi}G_{K-\ell,\overline{\chi}}\rangle_{D}=\frac{\Gamma(K-1)(K-\ell-1)!D^{K-\ell}\overline{L(K-\ell,\chi)}}{(4\pi)^{K-1}(2\pi i)^{K-\ell}\overline{G(\chi)}}\sum_{n=1}^{\infty}\frac{a_f(n){\sigma_{\ell-1,\overline{\chi}}(n)}}{n^{K-1}}.\label{eq:step1RS}
\end{align}
By Lemma \ref{lem:easycomp}, we have 
    \begin{align}
\sum_{n=1}^{\infty}\frac{a_f(n)\sigma_{\ell-1,\overline{\chi}}(n)}{n^{K-1}}=\frac{L(f,K-1)L(f,\overline{\chi},K-\ell)}{L(K-\ell,\overline{\chi})}.\label{eq:RSstep2}
    \end{align}
    Combining \eqref{eq:step1RS} and \eqref{eq:RSstep2} gives the result. 
\end{proof}
\begin{lemma}\label{lem:easycomp}
  Let $f(z)=\sum_{n\geq1}a_{f}(n)q^n\in S_K$ be a normalized Hecke eigenform. Then for a Dirichlet character $\chi$, an integer $\ell\geq3$ and a complex number $s$ with $\re(s)>\ell+\frac{K-1}{2}$, we have 
  \begin{align*} 
        \sum_{n=1}^{\infty}\frac{a_f(n)\sigma_{\ell-1,\chi}(n)}{n^{s}}=\frac{L(f,s)L(f,\chi,s-\ell+1)}{L(2s-\ell+2-K,\chi)}.
  \end{align*}
\end{lemma}
\begin{proof}
    Since $f\in S_K$ is a normalized Hecke eigenform, we have 
    \begin{align*}
        a_f(m)a_f(n)=\sum_{d\mid(m,n)}a_f\left(\frac{mn}{d^2}\right)d^{K-1}.
    \end{align*}
Note that $s$ is within the region of convergence for all the involved series. Using above Hecke relation, we get
\begin{align}
    L(f,s)L(f,\chi,s-\ell+1)
    &=\sum_{m,n\geq1}a_f(m)a_f(n)\chi(n)(mn)^{-s}n^{\ell-1} \nonumber\\&=\sum\limits_{\substack{m,n\geq1\\d\mid(m,n)}}\chi(n)a_{f}\left(\frac{mn}{d^2}\right)d^{K-1}(mn)^{-s}n^{\ell-1} \nonumber\\&=\sum_{d\geq1}\chi(d)d^{K-1+\ell-1-2s}\sum_{m,n\geq1}\chi(n)a_{f}(mn)(mn)^{-s}n^{\ell-1} \nonumber\\&=L(2s-\ell+2-K,\chi)\sum_{n=1}^{\infty}\frac{a_{f}(n)\sigma_{\ell-1,\chi}(n)}{n^s}, \nonumber
\end{align}
where the second to last inequality comes from the change of variables $m\mapsto dm$ and $n\mapsto dn$.
\end{proof}

\begin{proposition}\label{prop:Rankin-selbergperiod}
   Let $3\leq\ell\leq \frac{K-2}{2}$ be integers and $\chi$ be a primitive Dirichlet character mod $D$ such that $\chi(-1)=(-1)^{\ell}$, and $f\in S_K$ be a normalized Hecke eigenform. Then
    \begin{align}
        \langle f,\mathcal{F}_{K,\ell,\chi}\rangle=\frac{\Gamma(K-1)(K-\ell-1)!D^{K-\ell}}{(4\pi)^{K-1}(2\pi i)^{K-\ell}\overline{G(\chi)}}L(f,K-1)L(f,\overline{\chi},K-\ell),
        \label{eq:convolution}
    \end{align}
    where $\mathcal{F}_{K,\ell,\chi}$ is defined in \eqref{eq:F-def}.
\end{proposition}
\begin{proof}
    Since $\langle f, G_{K}\rangle=0$, and $\langle f,g\rangle_{M}=\langle f,\Tr_N^M g\rangle_{N}$ for $N\mid M$, $f\in S_K(N)$, $g\in M_{K}(M)$ (see \cite[p.~271]{GZ-86}), we have
    \begin{align}
    \langle f,\mathcal{F}_{K,\ell,\chi}\rangle_1&=\langle f,\Tr_1^D(G_{\ell,\chi}(z)G_{K-\ell,\overline{\chi}}(z))\rangle_1 \nonumber\\&= \langle f,G_{\ell,\chi}G_{K-\ell,\overline{\chi}}\rangle_{D}. \nonumber
    \end{align}
    Now the result follows from Proposition \ref{prop:rankinselbergconvolution}.
\end{proof}
\begin{remark}\label{remark:centralprod}
    We would like to mention that Proposition \ref{prop:rankinselbergconvolution} (and thus Proposition \ref{prop:Rankin-selbergperiod}) also holds true for $\ell=\frac{K}{2}$, see \cite[p.~146]{Zagier1976} and \cite[Proposition 4.1]{prodofEisenstein2018}.
\end{remark}
To tackle  Theorems \ref{thm:mainthmfixDdifferentell2} and \ref{thm:twistbydifferentD2}, we need a similar inner product formula for a normalized Hecke eigenform $f$ and Rankin-Cohen brackets of Eisenstein series. 

\begin{proposition}\label{prop:rankinselbergconvolution2}
 Let $3\leq\ell\leq\frac{K-4}{2}$ and $k=K-2-\ell$ be integers, $\chi$ a primitive Dirichlet character mod $D$ such that $\chi(-1)=(-1)^{\ell}$, and $f(z)=\sum_{n\geq1}a_f(n)q^n\in S_K$ be a normalized Hecke eigenform. Then
    \begin{align}
         \langle f,[G_{\ell,\chi},G_{k,\overline{\chi}}]_1\rangle_{D}=\frac{-k!(K-2)!D^{k}}{(4\pi)^{K-1}(2\pi i)^{k}\overline{G(\chi)}}L(f,K-2)L(f,\overline{\chi},K-\ell-1). \nonumber
     \end{align}
\end{proposition}
\begin{proof}
Again, by Lemma \ref{lem:RankinSelbergZagier} (for $k_1=\ell, k_2=k$), \eqref{eq:EisensteinforR-S} and the fact that $\overline{\sigma_{\ell-1,\chi}(n)}=\sigma_{\ell-1,\overline{\chi}}(n)$, we get
    \begin{align}
        \langle f,[G_{\ell,\chi},G_{k,\overline{\chi}}]_1\rangle_{D}&=-\frac{\Gamma(K-1)\Gamma(k+1)}{(4\pi)^{K-1}\Gamma(k)}\cdot\frac{(k-1)!D^k\overline{L(k,\chi)}}{(2\pi i)^{k}\overline{G(\chi)}}\sum_{n=1}^{\infty}\frac{a_f(n)\sigma_{\ell-1,\overline{\chi}}(n)}{n^{K-2}}.\label{eq:11}
    \end{align}
As $K-2>\ell+\frac{K-1}{2}$, we have  
    \begin{align}
        \sum_{n=1}^{\infty}\frac{a_f(n)\sigma_{\ell-1,\chi}(n)}{n^{K-2}}&=\frac{L(f,K-2)L(f,\overline{\chi},K-2-\ell+1)}{L(2(K-2)-\ell+2-K,\overline{\chi})} \nonumber\\&=\frac{L(f,K-2)L(f,\overline{\chi},K-\ell-1)}{L(k,\overline{\chi})},\label{eq:22}
    \end{align}
by Lemma \ref{lem:easycomp}. Plugging   \eqref{eq:22} into \eqref{eq:11} gives the result. 
\end{proof}
The following proposition generalizes \cite[Proposition 1]{Kohnen-Zagier1981} to the  Rankin-Cohen bracket.
\begin{proposition}\label{prop:Rankin-selbergperiod2}
   Under the hypothesis of Proposition \ref{prop:rankinselbergconvolution2}, we have 
    \begin{align}
        \langle f,\mathcal{G}_{K,\ell,\chi}\rangle=\frac{-k!(K-2)!D^{k}}{(4\pi)^{K-1}(2\pi i)^{k}\overline{G(\chi)}}L(f,K-2)L(f,\overline{\chi},K-\ell-1).\label{eq:convolution2}
    \end{align}
\end{proposition}
\begin{proof}
    By definition \eqref{eq:G-def},
 $\langle f,\mathcal{G}_{K,\ell,\chi}\rangle=\langle f,\Tr_1^D[G_{\ell,\chi},G_{k,\overline{\chi}}]_1\rangle_1= \langle f,[G_{\ell,\chi},G_{K-\ell,\overline{\chi}}]_1\rangle_{D}.
    $
    So the result follows from Proposition \ref{prop:rankinselbergconvolution2}. 
\end{proof}
We now establish the equivalence between the linear independence of twisted periods and the linear independence of the corresponding cusp forms $\mathcal{F}_{K,\ell,\chi}$ or $\mathcal{G}_{K,\ell,\chi}$.
\begin{proposition}\label{prop:linearindeperiodtomodularform}
Let $\chi$ be a primitive Dirichlet character mod $D$. Let $n\geq1$ and $3\leq\ell_1<\ell_2<...<\ell_n\leq\frac{K-2}{2}$ be integers such that $\chi(-1)=(-1)^{\ell_i}$ for all $1\leq i\leq n$. Then
$\{\mathcal{F}_{K,\ell_i,\chi}\}_{i=1}^n$ is linearly independent if and only if $\{r_{\ell_i-1,\chi}\}_{i=1}^n$ is linearly independent.
\end{proposition}
\begin{proof}
Note that Proposition \ref{prop:Rankin-selbergperiod} together with the definition of twisted period \eqref{eq:periodandLvalue}  and the functional equation \eqref{eq:functionaleqoftwistedL} implies that for $3\leq\ell\leq\frac{K-2}{2}$ and a normalized Hecke eigenform $f\in S_K$
\begin{align}
\langle f, \mathcal{F}_{K,\ell,\chi}\rangle=A_{K,\ell,\chi}L(f,K-1) r_{\ell-1,\chi}(f), \nonumber
\end{align}
 where $A_{K,\ell,\chi}$ is some nonzero constant depending only on $K$, $\ell$ and $\chi$, and $L(f,K-1)\neq0$ since $K-1$ is within the region of absolute convergence for $L(f,s)$. Let $\mathcal{H}_K$ denote the set of normalized Hecke eigenform in $S_K$. Then
\begin{align}\sum_{i=1}^na_i\mathcal{F}_{K,\ell_i,\chi}=0\quad{\rm if~and~only~if}\quad\sum_{i=1}^n\overline{a_i}\langle f_j,\mathcal{F}_{K,\ell_i,\chi}\rangle=0\quad{\rm for~all}~f_j\in\mathcal{H}_K.\label{eq:suminnerproduct} \nonumber\end{align} 
Suppose that $\{\mathcal{F}_{K,\ell_i,\chi}\}_{i=1}^n$ is linearly independent. We claim that $\{r_{\ell_i-1,\chi}\}_{i=1}^n$
is linearly independent. If $\sum_{i=1}^n b_ir_{\chi,\ell_{i}-1}=0\in S_K^{\ast},$ then
\begin{align}
    \sum_{i=1}^nb_i r_{\ell_i-1,\chi}(f_j)=L(f_j,K-1)^{-1}\sum_{i=1}^nb_iA_{K,\ell_i,\chi}^{-1}\langle f_j,\mathcal{F}_{K,\ell_i,\chi}\rangle=0~{\rm for~all}~f_j\in\mathcal{H}_K, \nonumber
\end{align}
which implies that
$\sum_{i=1}^n\overline{b_iA_{K,\ell_i,\chi}^{-1}}\mathcal{F}_{K,\ell_i,\chi}=0\in S_K$.
Since $\{\mathcal{F}_{K,\ell_i,\chi}\}_{i=1}^n$ is linearly independent, we have $b_iA_{K,\ell_i,\chi}^{-1}=0$, implying $b_i=0$ for all $1\leq i\leq n$. 

Conversely, suppose that $\{r_{\ell_i-1,\chi}\}_{i=1}^n$
is linearly independent. We show that $\{\mathcal{F}_{K,\ell_i,\chi}\}_{i=1}^n$ is linearly independent.  If
$\sum_{i=1}^n a_i\mathcal{F}_{K,\ell_i,\chi}=0\in S_K,$
then
\begin{align}&\sum_{i=1}^n\overline{a_i}\langle f_j,\mathcal{F}_{K,\ell_i,\chi}\rangle=0~{\rm for~all}~f_j\in\mathcal{H}_K, \nonumber\end{align}
which implies that $L(f_j,K-1)\sum_{i=1}^n\overline{a_i}A_{K,\ell_i,\chi}\cdot r_{\ell_i-1,\chi}(f_j)=0~{\rm for~all}~f_j\in\mathcal{H}_K$.
Since $L(f_j,K-1)\neq0$ for all $f_j\in\mathcal{H}_K$ and $\{r_{\ell_i-1,\chi}\}_{i=1}^n$
is linearly independent, we get $\overline{a_i}A_{K,\ell_i,\chi}=0$, and thus $a_i=0$ for all $1\leq i\leq n$. This completes the proof.
\end{proof}
\begin{proposition}\label{prop:linearindeperiodtomodularform2}
Let $\chi$ be a primitive Dirichlet character mod $D$. Let $n\geq1$ and $3\leq\ell_1<\ell_2<...<\ell_n\leq\frac{K-4}{2}$ be integers such that $\chi(-1)=(-1)^{\ell_i}$ for all $1\leq i\leq n$. Then
$\{\mathcal{G}_{K,\ell_i,\chi}\}_{i=1}^n$ is linearly independent if and only if $\{r_{\ell_i,\chi}\}_{i=1}^n$ on $S_K$ is linearly independent.
\end{proposition}
\begin{proof}
    Note that Proposition \ref{prop:Rankin-selbergperiod2} together with the definition of twisted period \eqref{eq:periodandLvalue}  and the functional equation \eqref{eq:functionaleqoftwistedL} implies that for $3\leq\ell\leq\frac{K-4}{2}$ and a normalized Hecke eigenform $f\in S_K$ we have
\begin{align}
\langle\mathcal{G}_{K,\ell,\chi},f\rangle=B_{K,\ell,\chi}L(f,K-2) r_{\ell,\chi}(f), \nonumber
\end{align}
 where $B_{K,\ell,\chi}$ is some nonzero constant depending only on $K$, $\ell$ and $\chi$, and $L(f,K-2)\neq0$. Then the proof proceeds in a similar way as the previous one.
\end{proof}
The following two propositions can be proved in the same way. We omit their proofs.
\begin{proposition}
   Let $3\leq\ell\leq\frac{K-2}{2}$ be an integer, and $\chi_1,...,\chi_n$ be primitive Dirichlet characters mod $D$ such that $\chi_i(-1)=(-1)^{\ell}$ and $\chi_i(2)$ are pairwise distinct for $1\leq i\leq n$. Then $\{\mathcal{F}_{K,\ell,\chi_i}\}_{i=1}^n$ is linearly independent if and only if $\{r_{\ell-1,\chi_i}\}_{i=1}^n$ on $S_K$ is linearly independent.
\end{proposition}
\begin{proposition}
   Let $3\leq\ell\leq\frac{K-4}{2}$ be an integer, and $\chi_1,...,\chi_n$ be primitive Dirichlet characters mod $D$ such that $\chi_i(-1)=(-1)^{\ell}$ and $\chi_i(2)$ are pairwise distinct for $1\leq i\leq n$. Then $\{\mathcal{G}_{K,\ell,\chi_i}\}_{i=1}^n$ is linearly independent if and only if $\{r_{\ell,\chi_i}\}_{i=1}^n$ on $S_K$ is linearly independent.
\end{proposition}


\section{Fourier coefficients}\label{sect:Fourier}
In this section, we compute the Fourier coefficients of $\mathcal{F}_{K,\ell,\chi}$ \eqref{eq:F-def}, $\mathcal{G}_{K,\ell,\chi}$ \eqref{eq:G-def} and study their asymptotic behaviors.  First of all, we need an explicit formula for the Eisenstein series $G_{k,\chi}$ under the action of certain matrices in $\Sl_2(\mathbb{Z})$. 
\subsection{Fourier expansion of Eisenstein series at cusps}
We need to introduce another type of Eisenstein series, $G_{k,\chi_1,\chi_2}(z)$, defined below. 
 Let $D=D_1D_2$ be such that $D_1>0$. Then there is a unique decomposition $\chi=\chi_1\chi_2$, where $\chi_1$ and $\chi_2$ are primitive Dirichlet characters mod $D_1$ and mod $D_2$, respectively. Define 
\begin{align}
    G_{k,\chi_1,\chi_2}(z)&:=\frac{D_1^k(k-1)!}{2(-2\pi i)^k G(\overline{\chi}_{1})}E_{k}(D_1z;\chi_{2},\overline{\chi}_{1})
    \\&=\frac{D_1^k(k-1)!}{2(-2\pi i)^k G(\overline{\chi}_{1})}\chi_{2}(D_1)\sideset{}{'}\sum_{\substack{m,n\in\mathbb{Z}\\D_1\mid m}}\frac{\chi_{2}(m)\overline{\chi}_{1}(n)}{(mz+n)^k}.\label{eq:GkD1D2}
\end{align}
We would like to remark that such definition is for the consistency of notation in section \ref{subsect:computationofFourier}, although it is not necessarily consistent with those of $E_k(z;\chi,\psi)$.
From Lemma \ref{lem:EisensteinFourierexpansion} we know that
\begin{align}
    E_{k}(z;\chi_{2},\overline{\chi}_{1}):&=C+\frac{2(-2\pi i)^kG(\overline{\chi}_{1})}{D_1^k(k-1)!}\sum_{n=1}^{\infty}\left(\sum_{\substack{d_1,d_2>0\\ d_1d_2=n}}\chi_1(d_1)\chi_2(d_2)d_1^{k-1}\right)e^{2\pi inz/D_1},  
    \nonumber\\C:&=\begin{cases}
        2L(k,\overline{\chi}_{1})&D_2=1,\\0&{\rm otherwsie}.
    \end{cases}   \nonumber
\end{align}
Thus, we have 
\begin{align}
    \label{eq:gkd1d2def}
    G_{k,\chi_1,\chi_2}(z)&=\sum_{n\geq0}\sigma_{k-1,\chi_1,\chi_2}(n)q^n,\\\
    \sigma_{k-1,\chi_1,\chi_2}(n)&:=\begin{cases}0&n=0{\rm~ and}~D_2\neq1,\\
       \frac{(k-1)!D^k}{(-2\pi i)^kG(\overline{\chi})}L(k,\overline{\chi}) &n=0~{\rm and}~D_2=1,\\\sum\limits_{\substack{d_1,d_2>0\\ d_1d_2=n}}\chi_1(d_1)\chi_2(d_2)d_1^{k-1} &n\geq1.
    \end{cases}\label{eq:fouriercochi1chi2}
\end{align}

We now compute the Fourier expansion of Eisenstein series $G_{k,\chi}(z)$ under the action of certain matrices in $\Sl_2(\mathbb{Z})$. 
\begin{lemma}\label{lem:FourierexpansionofGkD} Let $D=D_1D_2$ with $D_1>0$ 
and $\chi$ be a primitive Dirichlet character mod $D$ such that $\chi(-1)=(-1)^k$. If $\begin{bsmallmatrix}
    a&b\\c&d
\end{bsmallmatrix}$ is a matrix in $\Sl_2(\mathbb{Z})$ such that $\gcd(c,D)=D_1$ then
\begin{align}
    \left(G_{k,\chi}\bigg|_k\begin{bmatrix}
        a & b\\ c& d\end{bmatrix}\right)(z)=\chi_2(c)\chi_1(d)\overline{\chi}_{1}(D_2)\overline{\chi}_{2}(D_1)\frac{G(\overline{\chi}_{1})}{G(\overline{\chi})}G_{k,\chi_1,\overline{\chi}_{2}}\left(\frac{z+c^{\ast}d}{D_2}\right),   \nonumber
\end{align}
where $\chi=\chi_1\chi_2$ is the product of primitive characters $\chi_1$ mod $D_1$ and $\chi_2$ mod $D_2$, and $c^{\ast}$ is an integer with $cc^{\ast}\equiv 1\pmod{D_2}$ and $D_1\mid c^{\ast}$. 
\end{lemma}
\begin{proof}
     We follow the idea in \cite[pp.~273-275]{GZ-86}.   By equation \eqref{eq:GkDforcomputingFourier}, we have 
    \begin{align}
    \left(\frac{2(-2\pi i)^kG(\overline{\chi})}{(k-1)!D^k}G_{k,\chi}\bigg|_k\begin{bmatrix}
        a&b\\ c&d\end{bmatrix}\right)(z)&=\sideset{}{'}\sum_{\substack{l,r\in\mathbb{Z}\\D\mid l}}\frac{\overline{\chi}(r)}{(l(az+b)+r(cz+d))^k}   \nonumber\\&=\sideset{}{'}\sum_{\substack{l,r\in\mathbb{Z}\\D\mid l}}\frac{\overline{\chi}(r)}{((al+cr)z+bl+dr)^k}   \nonumber\\&=\sideset{}{'}\sum_{\substack{m,n\in\mathbb{Z} \\md\equiv nc~\text{mod}~D}}\frac{\overline{\chi}(an-bm)}{(mz+n)^k},   \nonumber
    \end{align}
    where $(m,n)=(l,r)\begin{bsmallmatrix}
        a&b\\ c&d\end{bsmallmatrix}$ and thus $r=an-bm$. Since $md\equiv nc\pmod{D}$, we have 
        \begin{align}
            d(an-bm)&=adn-bmd\equiv adn-bcn\equiv n\pmod{D},\label{eq:dealwithchiD1}\\ c(an-bm)&=anc-bcm\equiv adm-bcm\equiv m\pmod{D}.\label{eq:dealwithchiD2}
        \end{align}
Note also that $\gcd(D_1,D_2)=1$. Then \eqref{eq:dealwithchiD1} and \eqref{eq:dealwithchiD2} imply that
 \begin{align}
     \overline{\chi}(an-bm)&=\overline{\chi}_{1}(an-bm)\overline{\chi}_{2}(an-bm)   \nonumber\\&=\chi_{1}(d)\overline{\chi}_{1}(n)\chi_{2}(c)\overline{\chi}_{2}(m).   \nonumber
 \end{align}
 Since $D_1,D_2\mid(md-nc)$, $(d,D_1)=1, (c,D_2)=1$ and $(c,D)=D_1$, we must have $D_1\mid m$; and $n\equiv c^{\ast}md\pmod{D_2}$. Replacing $n$ by $n=c^{\ast}md+lD_2$, and choosing  $c^{\ast}$ to satisfy $D_1\mid c^{\ast}$ by the Chinese Remainder Theorem, so that $\overline{\chi}_1(c^{\ast}md+lD_2)=\overline{\chi}_1(l)\overline{\chi}_1(D_2)$. It follows that
 \begin{align}
 &\left(\frac{2(-2\pi i)^kG(\overline{\chi})}{(k-1)!D^k}G_{k,\chi}\bigg|_k\begin{bmatrix}
        a&b\\ c&d\end{bmatrix}\right)(z)   \nonumber\\=&\sideset{}{'}\sum_{\substack{m,l\in\mathbb{Z}\\ D_1\mid m}}\frac{\chi_{1}(d)\overline{\chi}_{1}(c^{\ast}md+lD_2)\chi_{2}(c)\overline{\chi}_{2}(m)}{(mz+mc^{\ast}d+lD_2)^k}   \nonumber\\=&\chi_{2}(c)\chi_{1}(d)\overline{\chi}_{1}(D_2)\sideset{}{'}\sum_{\substack{m,l\in\mathbb{Z}\\ D_1\mid m}}\frac{\overline{\chi}_{2}(m)\overline{\chi}_{1}(l)}{(mz+mc^{\ast}d+lD_2)^k}   \nonumber\\=&\chi_{2}(c)\chi_{1}(d)\overline{\chi}_{1}(D_2)D_2^{-k}\sideset{}{'}\sum_{\substack{m,l\in\mathbb{Z}\\ D_1\mid m}}\frac{\overline{\chi}_{2}(m)\overline{\chi}_{1}(l)}{\left(m\frac{z+c^{\ast}d}{D_2}+l\right)^k}.\label{eq:middle1}
 \end{align}
Note that \eqref{eq:GkD1D2} implies that 
\begin{align}
    \sideset{}{'}\sum_{\substack{m,l\in\mathbb{Z}\\ D_1\mid m}}\frac{\overline{\chi}_{2}(m)\overline{\chi}_{1}(l)}{\left(m\frac{z+c^{\ast}d}{D_2}+l\right)^k}=\frac{2(-2\pi i)^kG(\overline{\chi}_{1})}{D_1^k(k-1)!}\overline{\chi}_{2}(D_1)G_{k,\chi_1,\overline{\chi}_{2}}\left(\frac{z+c^{\ast}d}{D_2}\right).\label{eq:middle2}
\end{align} 
Plugging \eqref{eq:middle2} into \eqref{eq:middle1} gives the desired result.
\end{proof}
\subsection{Computation of the trace}\label{subsect:computationofFourier}
For $m\geq1$ and $f(z)=\sum_{n\geq0}a_f(n)q^n\in S_{k}(N,\chi)$ we define the $U$-operator
\begin{align}
    U_mf(z)=\frac{1}{m}\sum_{v~{\rm mod}~m}f\left(\frac{z+v}{m}\right)=\sum_{n\geq1}a_f(mn)q^n.\label{eq:defofUmap}
\end{align}
Equivalently, we may write 
\begin{align}U_mf(z)=m^{k/2-1}\sum_{v\text{ mod }m}f(z)\bigg|_k\begin{bmatrix}
    1 & v\\ 0& m
\end{bmatrix}.\label{eq:defofUslash}\end{align} 
We are ready to compute the Fourier coefficients of $\mathcal{F}_{K,\ell,\chi}(z)$ \eqref{eq:F-def}. 
\begin{proposition}\label{prop:Trislifting}
    Let $3\leq \ell\leq \frac{K-2}{2}$ be integers and  $\chi$ be primitive character mod $D$ such that $\chi(-1)=(-1)^{\ell}$. Then
    \begin{align}
        \Tr^D_1(G_{\ell,\chi}(z)G_{K-\ell,\overline{\chi}}(z))=\sum_{D=D_1D_2}\overline{\chi}_2(-1)U_{D_2}(G_{\ell,\chi_1,\overline{\chi}_{2}}(z)G_{K-\ell,\overline{\chi}_{1},\chi_2}(z)),   \nonumber
    \end{align}
    where the summation is over all decompositions of $D=D_1D_2$ as a product of two positive integers and $\chi=\chi_1\chi_2$ is the product of primitive characters $\chi_1$ mod $D_1$ and $\chi_2$ mod $D_2$.
\end{proposition}
\begin{proof}
Let $k=K-\ell$. We consider the following system of representatives  of $\Gamma_0(D)\backslash \Sl_{2}(\mathbb{Z})$ (see \cite[p.~276]{GZ-86} and \cite[Lemma 3.1]{kayath2024subspacesspannedeigenformsnonvanishing}):
    \begin{equation*}
        \left\{\begin{bmatrix}
            1 & 0 \\ D_1 & 1
        \end{bmatrix}\begin{bmatrix}
            1 & \mu \\ 0 & 1
        \end{bmatrix}~:~ D=D_1D_2,~\mu \text{ mod }D_2\right\}.
    \end{equation*}
 By Lemma \ref{lem:FourierexpansionofGkD}, we have 
    \begin{align}
    G_{\ell,\chi}(z)\bigg|_{k}\begin{bmatrix}
            1 & 0 \\ D_1 & 1
        \end{bmatrix}\begin{bmatrix}
            1 & \mu \\ 0 & 1
        \end{bmatrix} &=\chi_2(D_1)\overline{\chi}_{1}(D_2)\overline{\chi}_{2}(D_1)\frac{G(\overline{\chi}_{1})}{G(\overline{\chi})}G_{\ell,\chi_1,\overline{\chi}_{2}}\left(\frac{z+\mu+D_1^*}{D_2}\right),   \nonumber\\
       G_{k,\overline{\chi}}(z)\bigg|_{k}\begin{bmatrix}
            1 & 0 \\ D_1 & 1
        \end{bmatrix}\begin{bmatrix}
            1 & \mu \\ 0 & 1
        \end{bmatrix} &=\overline{\chi}_{2}(D_1)\chi_1(D_2)\chi_2(D_1)\frac{G(\chi_{1})}{G(\chi)}G_{k,\overline{\chi}_{1},\chi_2}\left(\frac{z+\mu+D_1^*}{D_2}\right),   \nonumber
    \end{align}
where $D_1^{\ast}D_1\equiv1\pmod {D_2}$ and $D_1\mid D_1^{\ast}$. 
It follows that 
    \begin{align*}\Tr_1^D(G_{\ell,\chi}(z)G_{k,\overline{\chi}}(z)) 
        =&\sum_{D_1D_2=D}\sum_{\mu \text{ mod } D_2} G_{\ell,\chi}(z)G_{k,\overline{\chi}}(z)\bigg|_{K}\begin{bmatrix}
            1 & 0 \\ D_1 & 1
        \end{bmatrix}\begin{bmatrix}
            1 & \mu \\ 0 & 1
        \end{bmatrix}    \nonumber\\
        =&\sum_{D_1D_2=D}\sum_{\mu \text{ mod } D_2}G_{\ell,\chi}(z)\bigg|_{\ell}\begin{bmatrix}
            1 & 0 \\ D_1 & 1
\end{bmatrix}\begin{bmatrix}
            1 & \mu \\ 0 & 1
        \end{bmatrix}\cdot G_{k,\overline{\chi}}(z)\bigg|_{k}\begin{bmatrix}
            1 & 0 \\ D_1 & 1
        \end{bmatrix}\begin{bmatrix}
            1 & \mu \\ 0 & 1
        \end{bmatrix}   \nonumber\\
        =&\sum_{D_1D_2=D}\sum_{\mu \text{ mod } D_2}\overline{\chi}_2(-1)D_2^{-1}G_{\ell,\chi_1,\overline{\chi}_{2}}\left(\frac{z+\mu+D_1^*}{D_2}\right)G_{k,\overline{\chi}_{1},\chi_2}\left(\frac{z+\mu+D_1^*}{D_2}\right),   \nonumber
    \end{align*}
where we used the fact
\begin{align*}G(\chi_{1})G(\overline{\chi}_{1})=\chi_1(-1)D_1\quad {\rm and}\quad G(\chi)G(\overline{\chi})=\chi(-1)D\end{align*} in the last equality (see e.g \cite[Corollary 2.1.47 on p.~33]{Cohenbook}). On the other hand,
    \begin{align*}
U_{D_2}\left(G_{\ell,\chi_1,\overline{\chi}_{2}}(z)G_{k,\overline{\chi}_{1},\chi_2}(z)\right)
        =&\sum_{v~\text{mod}~D_2}D_2^{K/2-1}G_{\ell,\chi_1,\overline{\chi}_{2}}(z)G_{k,\overline{\chi}_{1},\chi_2}(z)\bigg|_{K}\begin{bmatrix}
            1 & v\\ 0 & D_2
        \end{bmatrix}\\=&\sum_{v~\text{mod}~D_2}D_2^{K/2-1}G_{\ell,\chi_1,\overline{\chi}_{2}}(z)\bigg|_{\ell}\begin{bmatrix}
            1 & v\\ 0 & D_2
\end{bmatrix}\cdot G_{k,\overline{\chi}_{1},\chi_2}(z)\bigg|_k\begin{bmatrix}1 & v\\ 0 & D_2
        \end{bmatrix}\\=&\sum_{v~\text{mod}~D_2}D_2^{K/2-1}D_2^{-\ell/2}G_{\ell,\chi_1,\overline{\chi}_{2}}\left(\frac{z+v}{D_2}\right) D_2^{-k/2}G_{k,\overline{\chi}_{1},\chi_2}\left(\frac{z+v}{D_2}\right)\\=&\sum_{v~\text{mod}~D_2}D_2^{-1}G_{\ell,\chi_1,\overline{\chi}_{2}}\left(\frac{z+v}{D_2}\right)G_{k,\overline{\chi}_{1},\chi_2}\left(\frac{z+v}{D_2}\right).
    \end{align*}
    It follows that 
    \begin{align*}
       &\quad \,\sum_{D=D_1D_2}\overline{\chi}_2(-1)U_{D_2}\left(G_{\ell,\chi_1,\overline{\chi}_{2}}(z)G_{k,\overline{\chi}_{1},\chi_2}(z)\right)\\&=\sum_{D=D_1D_2}\sum_{v~\text{mod}~D_2}\overline{\chi}_2(-1)D_2^{-1}G_{\ell,\chi_1,\overline{\chi}_{2}}\left(\frac{z+v}{D_2}\right)G_{k,\overline{\chi}_{1},\chi_2}\left(\frac{z+v}{D_2}\right),
    \end{align*}
    as desired.
\end{proof}
\begin{proposition}\label{prop:Fourco}
   Let $3\leq \ell\leq\frac{K-2}{2}$ be an integer 
   and  $\chi$ be a primitive character mod $D$ such that $\chi(-1)=(-1)^{\ell}$. Then the $q^n$-Fourier coefficient of $\Tr^D_1(G_{\ell,\chi}(z)G_{K-\ell,\overline{\chi}}(z)) $ is 
    \begin{align}
\sum_{D=D_1D_2}\overline{\chi}_2(-1)\sum_{\substack{a_1,a_2\geq0\\a_1+a_2=nD_2}}\sigma_{\ell-1,\chi_1,\overline{\chi}_{2}}(a_1)\sigma_{k-1,\overline{\chi}_{1},\chi_2}(a_2),\label{eq:fouriercoTr}
    \end{align}
    where $k=K-\ell$,  the summation is over all decompositions of $D=D_1D_2$ as a product of two positive integers and $\chi=\chi_1\chi_2$ is the product of primitive characters $\chi_1$ mod $D_1$ and $\chi_2$ mod $D_2$.
\end{proposition}
\begin{proof}
Note that the $q^n$-Fourier coefficient of $U_{D_2}\left(G_{\ell,\chi_1,\overline{\chi}_{2}}(z)G_{K-\ell,\overline{\chi}_{1},\chi_2}(z)\right)$ is the $q^{nD_2}$-Fourier coefficient of $G_{\ell,\chi_1,\overline{\chi}_{2}}(z)G_{K-\ell,\overline{\chi}_{1},\chi_2}(z)$, which is 
    \begin{align*}
\sum_{\substack{a_1,a_2\geq0\\a_1+a_2=nD_2}}\sigma_{\ell-1,\chi_1,\overline{\chi}_{2}}(a_1)\sigma_{k-1,\overline{\chi}_{1},\chi_2}(a_2).
    \end{align*}
Hence \eqref{eq:fouriercoTr} follows from Proposition \ref{prop:Trislifting}.
\end{proof}
\begin{corollary}\label{co:cuspproj}
   Let $3\leq \ell\leq\frac{K-2}{2}$ be an integer 
   and  $\chi$ be a primitive character mod $D$ such that $\chi(-1)=(-1)^{\ell}$. Then $\mathcal{F}_{K,\ell,\chi}$ is a cusp form of weight $K$ and level one.
\end{corollary}
\begin{proof}Let $k=K-\ell$.
    By Proposition \ref{prop:Fourco}, the constant term of  $\Tr^D_1(G_{\ell,\chi}(z)G_{K-\ell,\overline{\chi}}(z)) $  is 
\begin{align*}\sum_{D=D_1D_2}\chi_2(-1)^{-1}\sigma_{\ell-1,\chi_1,\overline{\chi}_{2}}(0)\sigma_{k-1,\overline{\chi}_{1},\chi_2}(0)=\sigma_{\ell-1,\chi,\mathbbm{1}}(0)\sigma_{k-1,\overline{\chi},\mathbbm{1}}(0).\end{align*}
Recall that the  Eisenstein series in $M_K$ is given by
$$G_K(z)=\frac{\zeta(1-K)}{2}+\sum_{n\geq1}\sigma_{K-1}(n)q^n.$$ 
It follows from its definition \eqref{eq:F-def} that
 \begin{align*}    \mathcal{F}_{K,\ell,\chi}(z)=\Tr_1^D(G_{\ell,\chi}(z)G_{K-\ell,\overline{\chi}}(z))-\frac{2\sigma_{\ell-1,\chi}(0)\sigma_{k-1,\overline{\chi}}(0)}{\zeta(1-K)}G_K(z)\in S_K,
 \end{align*}
 as desired.
\end{proof}
\begin{remark}\label{remark:traceofprodwhenl=k}
    Note that Propositions \ref{prop:Trislifting}, \ref{prop:Fourco} and Corollary \ref{co:cuspproj} also hold true for $\ell=\frac{K}{2}$ since the computation of the trace  is valid as long as $3\leq\ell\leq K-3$. See Kohnen-Zagier \cite[p. 193]{Kohnen-Zagier1981} for the case when $\chi$ is quadratic and $\ell=\frac{K}{2}$.
\end{remark}
We now compute the Fourier coefficients of $\mathcal{G}_{K,\ell,\chi}$, see \eqref{eq:G-def} for the definition.
\begin{proposition}\label{prop:TrisliftingRankinCohen}
     Let $3\leq \ell\leq\frac{K-4}{2}$ be an integer 
     and  $\chi$ be a primitive character mod $D$ such that $\chi(-1)=(-1)^{\ell}$. Then
    \begin{align*}
    \Tr^D_1([G_{\ell,\chi}(z),G_{K-2-\ell,\overline{\chi}}(z)]_1)=\sum_{D=D_1D_2}\overline{\chi}_2(-1)D_2^{-1}U_{D_2}\left([G_{\ell,\chi_1,\overline{\chi}_{2}}(z),G_{K-2-\ell,\overline{\chi}_{1},\chi_2}(z)]_1\right),
    \end{align*}
    where the summation is over all decompositions of $D$ as a product of two positive integers and $\chi=\chi_1\chi_2$ is the product of primitive characters $\chi_1$ mod $D_1$ and $\chi_2$ mod $D_2$.
\end{proposition}
\begin{proof}
Let $k=K-2-\ell$. Then  $\Tr^D_1([G_{\ell,\chi}(z),G_{K-2-\ell,\overline{\chi}}(z)]_1)=$    
    \begin{align*}
        =&\sum_{D_1D_2=D}\sum_{\mu \text{ mod } D_2} \left[G_{\ell,\chi}(z),G_{k,\overline{\chi}}(z)\right]_1\bigg|_{K}\begin{bmatrix}
            1 & 0 \\ D_1 & 1
        \end{bmatrix}\begin{bmatrix}
            1 & \mu \\ 0 & 1
        \end{bmatrix} \\
        =&\sum_{D_1D_2=D}\sum_{\mu \text{ mod } D_2}\left[G_{\ell,\chi}(z)\bigg|_{\ell}\begin{bmatrix}
            1 & 0 \\ D_1 & 1
\end{bmatrix}\begin{bmatrix}
            1 & \mu \\ 0 & 1
        \end{bmatrix},G_{k,\overline{\chi}}(z)\bigg|_{k}\begin{bmatrix}
            1 & 0 \\ D_1 & 1
        \end{bmatrix}\begin{bmatrix}
            1 & \mu \\ 0 & 1
        \end{bmatrix}\right]_1 \\
        =&\sum_{D_1D_2=D}\sum_{\mu \text{ mod } D_2}\overline{\chi}_2(-1)D_2^{-1}\left[G_{\ell,\chi_1,\overline{\chi}_{2}}\left(\frac{z+\mu+D_1^*}{D_2}\right),G_{k,\overline{\chi}_{1},\chi_2}\left(\frac{z+\mu+D_1^*}{D_2}\right)\right]_1,
    \end{align*}
    where $D_1^{\ast}D_1\equiv1\pmod{D_2}$ and $D_1\mid D_1^{\ast}$.
Note also that $U_{D_2}\left([G_{\ell,\chi_1,\overline{\chi}_{2}}(z),G_{k,\overline{\chi}_{1},\chi_2}(z)]_1\right)=  $
     \begin{align*}
     &\sum_{v~\text{mod}~D_2}D_2^{K/2-1}[G_{\ell,\chi_1,\overline{\chi}_{2}}(z),G_{k,\overline{\chi}_{1},\chi_2}(z)]_1\bigg|_{K}\begin{bmatrix}
            1 & v\\ 0 & D_2
    \end{bmatrix}\\=&\sum_{v~\text{mod}~D_2}D_2^{K/2-1}\left[G_{\ell,\chi_1,\overline{\chi}_{2}}(z)\bigg|_{\ell}\begin{bmatrix}
            1 & v\\ 0 & D_2
\end{bmatrix},G_{k,\overline{\chi}_{1},\chi_2}(z)\bigg|_k\begin{bmatrix}
            1 & v\\ 0 & D_2
\end{bmatrix}\right]_1\\=&\sum_{v~\text{mod}~D_2}D_2^{K/2-1}\left[D_2^{-\ell/2}G_{\ell,\chi_1,\overline{\chi}_{2}}\left(\frac{z+v}{D_2}\right),D_2^{-k/2}G_{k,\overline{\chi}_{1},\chi_2}\left(\frac{z+v}{D_2}\right)\right]_1\\=&\sum_{v~\text{mod}~D_2}\left[G_{\ell,\chi_1,\overline{\chi}_{2}}\left(\frac{z+v}{D_2}\right),G_{k,\overline{\chi}_{1},\chi_2}\left(\frac{z+v}{D_2}\right)\right]_1.
\end{align*}
It follows that 
\begin{align*}
&\sum_{D=D_1D_2}\overline{\chi}_2(-1)D_2^{-1}U_{D_2}\left([G_{\ell,\chi_1,\overline{\chi}_{2}}(z),G_{K-2-\ell,\overline{\chi}_{1},\chi_2}(z)]_1\right)
\\=&\sum_{D=D_1D_2}\sum_{v~\text{mod}~D_2}\overline{\chi}_2(-1)D_2^{-1}\left[G_{\ell,\chi_1,\overline{\chi}_{2}}\left(\frac{z+v}{D_2}\right),G_{k,\overline{\chi}_{1},\chi_2}\left(\frac{z+v}{D_2}\right)\right]_1,
\end{align*}
which gives the result.
\end{proof}

\begin{proposition}\label{prop:FourcoRankinCohen}
   Let $3\leq \ell\leq\frac{K-4}{2}$ be an integer 
   and  $\chi$ be primitive character mod $D$ such that $\chi(-1)=(-1)^{\ell}$. Then the $q^n$-Fourier coefficient of $ \Tr^D_1([G_{\ell,\chi}(z),G_{K-2-\ell,\overline{\chi}}(z)]_1)$ is 
    \begin{align}
\sum_{D=D_1D_2}\overline{\chi}_2(-1)D_2^{-1}\sum_{\substack{a_1,a_2\geq0\\a_1+a_2=nD_2}}\sigma_{\ell-1,\chi_1,\overline{\chi}_{2}}(a_1)\sigma_{k-1,\overline{\chi}_{1},\chi_2}(a_2)(\ell a_2-k a_1),\label{eq:fouriercoTrRankincohen}
    \end{align}
    where $k=K-\ell$, the summation is over all the decompositions of $D$ as a product of two positive integers, and $\chi=\chi_1\chi_2$ is the product of primitive characters $\chi_1$ mod $D_1$ and $\chi_2$ mod $D_2$.
\end{proposition}
\begin{proof}
By \eqref{eq:defofUmap}, we know that the $q^n$-Fourier coefficient of $U_{D_2}\left([G_{\ell,\chi_1,\overline{\chi}_{2}}(z),G_{K-2-\ell,\overline{\chi}_{1},\chi_2}(z)]_1\right)$ is the $q^{nD_2}$-Fourier coefficient of $[G_{\ell,\chi_1,\overline{\chi}_{2}}(z),G_{K-2-\ell,\overline{\chi}_{1},\chi_2}(z)]_1$. 
Note that 
\begin{align*}
    G_{\ell,\chi_1,\overline{\chi}_2}^{(r)}(z)=\sum_{n\geq0}n^r\sigma_{\ell-1,\chi_1,\overline{\chi}_2}(n)q^n \quad{\rm and}\quad  G_{k,\overline{\chi}_1,\chi_2}^{(1-r)}(z)=\sum_{n\geq0}n^{1-r}\sigma_{k-1,\overline{\chi}_1,\chi_2}(n)q^n, 
\end{align*}
which implies that the $q^{nD_2}$-Fourier coefficient of $[G_{\ell,\chi_1,\overline{\chi}_{2}}(z),G_{K-2-\ell,\overline{\chi}_{1},\chi_2}(z)]_1$ is 
\begin{align*}
&\sum_{0\leq r\leq1}(-1)^r\binom{\ell}{1-r}\binom{k}{r}\sum_{\substack{a_1,a_2\geq0\\a_1+a_2=nD_2}}a_1^r\sigma_{\ell-1,\chi_1,\overline{\chi}_2}(a_1)a_2^{1-r}\sigma_{k-1,\overline{\chi}_1,\chi_2}(a_2)\\=&\sum_{\substack{a_1,a_2\geq0\\a_1+a_2=nD_2}}\sigma_{\ell-1,\chi_1,\overline{\chi}_{2}}(a_1)\sigma_{k-1,\overline{\chi}_{1},\chi_2}(a_2)(\ell a_2-k a_1).
\end{align*}
Hence \eqref{eq:fouriercoTrRankincohen} follows from Proposition \ref{prop:TrisliftingRankinCohen}.
\end{proof}

\subsection{Asymptotics}
Having obtained the explicit formulas for the Fourier coefficients of $\mathcal{F}_{K,\ell,\chi}$ and $\mathcal{G}_{K,\ell,\chi}$, we now investigate their asymptotic behaviors.  Let $a_{K,\ell,\chi}(n)$ and $b_{K,\ell,\chi}(n)$ denote the $q^n$- Fourier coefficient of $\mathcal{F}_{K,\ell,\chi}(z)$ \eqref{eq:F-def} and $\mathcal{G}_{K,\ell,\chi}(z)$ \eqref{eq:G-def}, respectively. We first normalize $\mathcal{F}_{K,\ell,\chi}(z)$ so that its $q$-Fourier coefficient becomes $1$. Define
 \begin{align}
     \mathfrak{F}_{K,\ell,\chi}(z):=\frac{\mathcal{F}_{K,\ell,\chi}(z)}{a_{K,\ell,\chi}(1)}. \label{eq:F-normalized}
 \end{align}
For $n\geq1$, let $\mathfrak{a}_{K,\ell,\chi}(n)$ denote the $q^n$-Fourier coefficient of $\mathfrak{F}_{K,\ell,\chi}(z)$. By Proposition \ref{prop:Fourco} and Corollary \ref{co:cuspproj}, we have 
\begin{align*} \mathfrak{a}_{K,\ell,\chi}(n)=&\frac{\sum\limits_{D=D_1D_2}\overline{\chi}_2(-1)\sum\limits_{\substack{a_1,a_2\geq0\\a_1+a_2=nD_2}}\sigma_{\ell-1,\chi_1,\overline{\chi}_{2}}(a_1)\sigma_{k-1,\overline{\chi}_{1},\chi_2}(a_2)-\frac{2\sigma_{\ell-1,\chi}(0)\sigma_{k-1,\overline{\chi}}(0)}{\zeta(1-K) }\sigma_{K-1}(n)}{\sum\limits_{D=D_1D_2}\overline{\chi}_2(-1)\sum\limits_{\substack{a_1,a_2\geq0\\a_1+a_2=D_2}}\sigma_{\ell-1,\chi_1,\overline{\chi}_{2}}(a_1)\sigma_{k-1,\overline{\chi}_{1},\chi_2}(a_2)-\frac{2\sigma_{\ell-1,\chi}(0)\sigma_{k-1,\overline{\chi}}(0)}{\zeta(1-K) }},
\end{align*}
where $k=K-\ell$. We rewrite the first term in the numerator. Note that
\begin{align*}
&\sum\limits_{D=D_1D_2}\overline{\chi}_2(-1)\sum\limits_{\substack{a_1,a_2\geq0\\a_1+a_2=nD_2}}\sigma_{\ell-1,\chi_1,\overline{\chi}_{2}}(a_1)\sigma_{k-1,\overline{\chi}_{1},\chi_2}(a_2)\\=
&\sum_{\substack{a_1,a_2\geq0\\a_1+a_2=n}}\sigma_{\ell-1,\chi,\mathbbm{1}}(a_1)\sigma_{k-1,\overline{\chi},\mathbbm{1}}(a_2)+\sum_{\substack{D=D_1D_2\\D_2\neq 1 }}\overline{\chi}_2(-1)\sum_{\substack{a_1,a_2\geq0\\a_1+a_2=nD_2}}\sigma_{\ell-1,\chi_1,\overline{\chi}_{2}}(a_1)\sigma_{k-1,\overline{\chi}_{1},\chi_2}(a_2)\\=&\sigma_{\ell-1,\chi}(n)\sigma_{k-1,\overline{\chi}}(0)+\sigma_{\ell-1,\chi}(0)\sigma_{k-1,\overline{\chi}}(n)+\sum_{a_1=1}^{n-1}\sigma_{\ell-1,\chi,\mathbbm{1}}(a_1)\sigma_{k-1,\overline{\chi},\mathbbm{1}}(n-a_1)\\&\quad+\sum_{\substack{D=D_1D_2\\D_2\neq 1 }}\overline{\chi}_2(-1)\sum_{\substack{a_1,a_2\geq0\\a_1+a_2=nD_2}}\sigma_{\ell-1,\chi_1,\overline{\chi}_{2}}(a_1)\sigma_{k-1,\overline{\chi}_{1},\chi_2}(a_2).
\end{align*}
Dividing  both the top and bottom by $\sigma_{k-1,\overline{\chi}}(0)$, we get that for $n\geq1$, 
\begin{align}
    \mathfrak{a}_{K,\ell,\chi}(n)=\frac{\sigma_{\ell-1,\chi}(n)+\frac{\sigma_{\ell-1,\chi}(0)}{\sigma_{k-1,\overline{\chi}}(0)}\sigma_{k-1,\overline{\chi}}(n)-\frac{2\sigma_{\ell-1,\chi}(0)}{\zeta(1-K) }\sigma_{K-1}(n)+\mathcal{E}_{K,\ell,n,\chi}+\mathscr{E}_{K,\ell,n,\chi}}{1+\frac{\sigma_{\ell-1,\chi}(0)}{\sigma_{k-1,\overline{\chi}}(0)}-\frac{2\sigma_{\ell-1,\chi}(0)}{\zeta(1-K) }+\mathscr{E}_{K,\ell,1,\chi}}, \label{eq:expressionoftheFoco}
\end{align}
where 
\begin{align*}
\mathcal{E}_{K,\ell,n,\chi}&=\sigma_{k-1,\overline{\chi}}(0)^{-1}\sum_{a_1=1}^{n-1}\sigma_{\ell-1,\chi}(a_1)\sigma_{k-1,\overline{\chi}}(n-a_1), \\
\mathscr{E}_{K,\ell,n,\chi}&=\sigma_{k-1,\overline{\chi}}(0)^{-1}\sum\limits_{\substack{D=D_1D_2\\D_2\neq 1}}\overline{\chi}_2(-1)\sum\limits_{\substack{a_1,a_2\geq0\\a_1+a_2=nD_2}}\sigma_{\ell-1,\chi_1,\overline{\chi}_{2}}(a_1)\sigma_{k-1,\overline{\chi}_{1},\chi_2}(a_2).
\end{align*}

 We also normalize $\mathcal{G}_{K,\ell,\chi}(z)$ such that its $q$-Fourier coefficient is $1.$ Define
\begin{align}
    \mathfrak{G}_{K,\ell,\chi}(z):=\frac{\mathcal{G}_{K,\ell,\chi}(z)}{b_{K,\ell,\chi}(1)}. \label{eq:G-normalized}
\end{align}
Denote by $\mathfrak{b}_{K,\ell,\chi}(n)$ the $q^n$-Fourier coefficient of $\mathfrak{G}_{K,\ell,\chi}(z)$.
By Proposition \ref{prop:FourcoRankinCohen}, we have 
    \begin{align*}
b_{K,\ell,\chi}(n)=&\sum_{D=D_1D_2}\overline{\chi}_2(-1)D_2^{-1}\sum_{\substack{a_1,a_2\geq0\\a_1+a_2=nD_2}}\sigma_{\ell-1,\chi_1,\overline{\chi}_{2}}(a_1)\sigma_{k-1,\overline{\chi}_{1},\chi_2}(a_2)(\ell a_2-k a_1)\\=&-kn\sigma_{\ell-1,\chi}(n)\sigma_{k-1,\overline{\chi}}(0)+\ell n\sigma_{k-1,\overline{\chi}}(n)\sigma_{\ell-1,\chi}(0)\\&+\sum_{\substack{
    a_1,a_2\geq1 \\a_1+a_2=n}}\sigma_{\ell-1,\chi}(a_1)\sigma_{k-1,\overline{\chi}}(a_2)(\ell a_2-ka_1)\\&+\sum_{\substack{D=D_1D_2\\D_2\neq1}}\overline{\chi}_2(-1)D_2^{-1}\sum_{\substack{a_1,a_2\geq0\\a_1+a_2=nD_2}}\sigma_{\ell-1,\chi_1,\overline{\chi}_{2}}(a_1)\sigma_{k-1,\overline{\chi}_{1},\chi_2}(a_2)(\ell a_2-k a_1).
    \end{align*}
After simplification, we get  
\begin{align}
   \mathfrak{b}_{K,\ell,\chi}(n)= \frac{n\sigma_{\ell-1,\chi}(n)-\frac{\sigma_{\ell-1,\chi}(0)}{\sigma_{k-1,\overline{\chi}}(0)}\frac{\ell}{k}n\sigma_{k-1,\overline{\chi}}(n)+\mathcal{R}_{K,\ell,n,\chi}+\mathcal{R}_{K,\ell,n,\chi}^{\prime}+\mathscr{R}_{K,\ell,n,\chi}+\mathscr{R}_{K,\ell,n,\chi}^{\prime}}{1-\frac{\sigma_{\ell-1,\chi}(0)}{\sigma_{k-1,\overline{\chi}}(0)}\frac{\ell}{k}+\mathscr{R}_{K,\ell,1,\chi}+\mathscr{R}^{\prime}_{K,\ell,1,\chi}},\label{eq:fouriercoefficientsb}
\end{align}
where
\begin{align*}
    \mathcal{R}_{K,\ell,n,\chi}&=\sigma_{k-1,\overline{\chi}}(0)^{-1}\sum_{a_1=1}^{n-1}a_1\sigma_{\ell-1,\chi}(a_1)\sigma_{k-1,\overline{\chi}}(n-a_1),\\
    \mathcal{R}_{K,\ell,n,\chi}^{\prime}&=-\frac{\ell}{k}\sigma_{k-1,\overline{\chi}}(0)^{-1}\sum_{a_1=1}^{n-1}\sigma_{\ell-1,\chi}(a_1)(n-a_1)\sigma_{k-1,\overline{\chi}}(n-a_1),\\
\mathscr{R}_{K,\ell,n,\chi}&=\sigma_{k-1,\overline{\chi}}(0)^{-1}\sum\limits_{\substack{D=D_1D_2\\D_2\neq 1}}\overline{\chi}_2(-1)D_2^{-1}\sum\limits_{\substack{a_1,a_2\geq0\\a_1+a_2=nD_2}}a_1\sigma_{\ell-1,\chi_1,\overline{\chi}_{2}}(a_1)\sigma_{k-1,\overline{\chi}_{1},\chi_2}(a_2),\\
\mathscr{R}^{\prime}_{K,\ell,n,\chi}&=-\frac{\ell}{k}\sigma_{k-1,\overline{\chi}}(0)^{-1}\sum\limits_{\substack{D=D_1D_2\\D_2\neq 1}}\overline{\chi}_2(-1)D_2^{-1}\sum\limits_{\substack{a_1,a_2\geq0\\a_1+a_2=nD_2}}\sigma_{\ell-1,\chi_1,\overline{\chi}_{2}}(a_1)a_2\sigma_{k-1,\overline{\chi}_{1},\chi_2}(a_2).
\end{align*}

The following inequality will be frequently used later.  
\begin{lemma}\label{lem:stirling}
    For $n\geq1$ we have 
    \begin{align*}
        \sqrt{2\pi}n^{n+1/2}e^{-n}<n!<en^{n+1/2}e^{-n}.
    \end{align*}
\end{lemma}
\begin{proof}
  The non-asymptotic Stirling's approximation \cite{Stirling's-formula} claims that for all $n\geq1$, we have 
\begin{align*} 
    \sqrt{2\pi n}\left(\frac{n}{e}\right)^ne^{\frac{1}{12n+1}}<n!<   \sqrt{2\pi n}\left(\frac{n}{e}\right)^ne^{\frac{1}{12n}},
\end{align*}
which implies the result.
\end{proof}
We also need the following trivial bounds for $L(k,\chi)$ for $k\geq2$.
\begin{lemma}\label{lem:boundforL}
    Let $k\geq2, N\geq1$ be integers, and $\chi$ be a Dirichlet character mod $N$. Then
    \begin{align*}
        2-\zeta(2)\leq2-\zeta(k)\leq|L(k,\chi)|\leq\zeta(k)\leq\zeta(2).
    \end{align*}
\end{lemma}
\begin{proof}
     Note that 
    $2-\zeta(2)\leq2-\zeta(k)=1-\sum_{n\geq2}n^{-k}\leq|L(k,\chi)|\leq \sum_{n\geq1}n^{-k}=\zeta(k)\leq\zeta(2).$
\end{proof}
\begin{lemma}\label{lem:bdquotientLvalues}
    Let $k\geq\ell\geq2$ be integers, $D\geq1$ an odd square-free integer, and $\chi$  a primitive Dirichlet character mod $D$. Then
   \begin{align*}
        \left|\frac{\sigma_{\ell-1,\chi}(0)}{\sigma_{k-1,\overline{\chi}}(0)}\right|\leq 6\left(\frac{\ell-1}{k-1}\right)^{\ell-1/2}\left(\frac{2\pi e}{(k-1)D}\right)^{k-\ell}.
         \end{align*}
\end{lemma}
\begin{proof}
By Lemmas \ref{lem:stirling} and \ref{lem:boundforL}, we have
\begin{align*}
    \left|\frac{\sigma_{\ell-1,\chi}(0)}{\sigma_{k-1,\overline{\chi}}(0)}\right|&=
    \left|\frac{\frac{(\ell-1)!D^{\ell}}{(-2\pi i)^{\ell}G(\overline{\chi})}L(\ell,\overline{\chi})}{\frac{(k-1)!D^k}{(-2\pi i)^kG(\chi)}L(k,\chi)}\right|
    \\&\leq\left|\frac{L(\ell,\overline{\chi})}{L(k,\chi)}\right|\frac{(\ell-1)^{\ell-1/2}e\cdot e^{-(\ell-1)}}{(k-1)^{k-1/2}\sqrt{2\pi}e^{-(k-1)}}\left(\frac{2\pi}{D}\right)^{k-\ell}\\&\leq\frac{\zeta(2)}{2-\zeta(2)}\cdot\frac{e}{\sqrt{2\pi}}\left(\frac{\ell-1}{k-1}\right)^{\ell-1/2}\left(\frac{2\pi e}{(k-1)D}\right)^{k-\ell}\\&\leq6\left(\frac{\ell-1}{k-1}\right)^{\ell-1/2}\left(\frac{2\pi e}{(k-1)D}\right)^{k-\ell},
\end{align*}
as desired.
\end{proof}
In the following several lemmas, we estimate the terms that appear in the expressions of Fourier coefficients \eqref{eq:expressionoftheFoco} and \eqref{eq:fouriercoefficientsb}.
\begin{lemma}\label{lem:bdquotientLandzeta}
    Let $3\leq\ell\leq\frac{K-2}{2}$ be an integer, $D\geq1$ an odd square-free integer, and $\chi$ a primitive Dirichlet character mod $D$. Then
    \begin{align*}
        \left|\frac{2\sigma_{\ell-1,\chi}(0)}{\zeta(1-K) }\right|\leq 2\left(\frac{(\ell-1)D}{K-1}\right)^{\ell-1/2}\left(\frac{2\pi e}{K-1}\right)^{K-\ell}.
    \end{align*}
\end{lemma}
\begin{proof}
    Again, by Lemmas \ref{lem:stirling} and \ref{lem:boundforL}, we get
    \begin{align}
          \left|\frac{2\sigma_{\ell-1,\chi}(0)}{\zeta(1-K) }\right|&=\left|\frac{2\frac{(\ell-1)!D^{\ell}}{(-2\pi i)^{\ell}G(\overline{\chi})}L(\ell,\overline{\chi})}{\frac{2(K-1)!\zeta(K)}{(2\pi)^K}}\right|
          \nonumber\\&\leq\frac{(\ell-1)^{\ell-1/2}e\cdot e^{-(\ell-1)}D^{\ell-1/2}(2\pi)^{K-\ell}\zeta(\ell)}{(K-1)^{K-1/2}\sqrt{2\pi}e^{-(K-1)}\zeta(K)}  \nonumber\\&=\frac{e\zeta(\ell)}{\sqrt{2\pi}}\left(\frac{(\ell-1)D}{K-1}\right)^{\ell-1/2}\left(\frac{2\pi e}{K-1}\right)^{K-\ell}\label{eq:accuratebdquotientzetaandL}\\&\leq\frac{e\zeta(2)}{\sqrt{2\pi}}\left(\frac{(\ell-1)D}{K-1}\right)^{\ell-1/2}\left(\frac{2\pi e}{K-1}\right)^{K-\ell}  \nonumber\\&\leq2\left(\frac{(\ell-1)D}{K-1}\right)^{\ell-1/2}\left(\frac{2\pi e}{K-1}\right)^{K-\ell},  \nonumber
    \end{align}
        as desired.
\end{proof}
\begin{lemma}\label{lem:bdsmalltermprod}
   Let $3\leq\ell\leq\frac{K-2}{2}$ and $n\ge1$ be integers, $k=K-\ell$, $D\geq1$ an odd square-free integer, and $\chi$ be a primitive Dirichlet character mod $D$. For $\mathcal{E}_{K,\ell,n,\chi}$ in \eqref{eq:expressionoftheFoco}, we have
    \begin{align*}
    |\mathcal{E}_{K,\ell,n,\chi}|\leq9.25\left(\frac{\pi en^2}{(K-2)D}\right)^{\frac{K-1}{2}}.
    \end{align*}
\end{lemma}
\begin{proof}
    Note that 
\begin{align}\sigma_{\alpha}(n):=\sum_{d\mid n}d^{\alpha}\leq\zeta(\alpha)n^{\alpha}\label{eq:anotherboundforsigma}\end{align} for all $\alpha>1$ and $n\geq1$, see \cite[p.~245]{Nathansonbook}. Then 
    \begin{align*}
        \left|\sum_{a_1=1}^{n-1}\sigma_{\ell-1,\chi}(a_1)\sigma_{k-1,\overline{\chi}}(n-a_1)\right|&\leq\zeta(\ell-1)\zeta(k-1)\sum_{ a_1=1}^{n-1}a_1^{\ell-1}(n-a_1)^{k-1}\\&\le\zeta(2)^2\sum_{ a_1=1}^{ n-1}a_1^{k-1}(n-a_1)^{k-1}\\&\leq\zeta(2)^2(n-1)\left(\frac{n}{2}\right)^{2k-2}\\&\leq2\zeta(2)^2\left(\frac{n}{2}\right)^{2k-1}.
    \end{align*}
    By Lemmas \ref{lem:stirling} and \ref{lem:bdquotientLvalues}, we have 
    \begin{align}
        \left|\sigma_{k-1,\overline{\chi}}(0)^{-1}\right|&=\frac{(2\pi)^k|G(\chi)|}{(k-1)!D^k|L(k,\chi)|}  \nonumber\\&\leq\frac{(2\pi)^k\sqrt{D}}{\sqrt{2\pi}(k-1)^{k-1/2}e^{-(k-1)}D^k(2-\zeta(k))}
        \nonumber  \\&=\frac{1}{\sqrt{e}(2-\zeta(k))}\left(\frac{2\pi e}{(k-1)D}\right)^{k-1/2} \label{eq:sharperbd2/L}
        \\&\leq\frac{1}{\sqrt{e}(2-\zeta(2))}\left(\frac{2\pi e}{(k-1)D}\right)^{k-1/2}.\label{eq:bound2/L}
    \end{align}
It follows that 
    \begin{align*}
\left|\mathcal{E}_{K,\ell,n,\chi}\right|&=\left|\sigma_{k-1,\overline{\chi}}(0)^{-1}\sum_{a_1=1}^{n-1}\sigma_{\ell-1,\chi}(a_1)\sigma_{k-1,\overline{\chi}}(n-a_1)\right|\\
         &\leq2\zeta(2)^2\left(\frac{n}{2}\right)^{2k-1}\cdot \frac{1}{\sqrt{e}(2-\zeta(2))}\left(\frac{2\pi e}{(k-1)D}\right)^{k-1/2}\\&=\frac{2\zeta(2)^2}{\sqrt{e}(2-\zeta(2))}\left(\frac{\pi en^2}{2(k-1)D}\right)^{k-1/2}\\&\leq 9.25\left(\frac{\pi en^2}{2(k-1)D}\right)^{k-1/2},
    \end{align*}
    as desired. 
\end{proof}
We now show that $\mathcal{R}_{K,\ell,n,\chi}$ and $\mathcal{R}_{K,\ell,n,\chi}^{\prime}$ in \eqref{eq:fouriercoefficientsb} have similar types of bounds as $\mathcal{E}_{K,\ell,n,\chi}$.
\begin{lemma}\label{lem:bdR}
   Let $3\leq\ell\leq\frac{K-4}{2}$ and $n\geq1$  be integers, $k=K-2-\ell$, $D\geq1$ an odd square-free integer, and $\chi$ be a primitive Dirichlet character mod $D$. Then
    \begin{align*}
    |\mathcal{R}_{K,\ell,n,\chi}|\leq9.25\left(\frac{\pi en^2}{(K-2)D}\right)^{\frac{K-1}{2}}.
    \end{align*}
\end{lemma}
\begin{proof}
By \eqref{eq:anotherboundforsigma}, we have
       \begin{align*}
        \left|\sum_{a_1=1}^{n-1}a_1\sigma_{\ell-1,\chi}(a_1)\sigma_{k-1,\overline{\chi}}(n-a_1)\right|&\leq\zeta(\ell-1)\zeta(k-1)\sum_{ a_1=1}^{n-1}a_1^{\ell}(n-a_1)^{k-1}\\&\leq\zeta(2)^2\sum_{ a_1=1}^{ n-1}a_1^{k-1}(n-a_1)^{k-1}\\&\leq2\zeta(2)^2\left(\frac{n}{2}\right)^{2k-1}.
    \end{align*}
Using the bound \eqref{eq:bound2/L}, we get    
\begin{align*}|\mathcal{R}_{K,\ell,n,\chi}|&=\left|\sigma_{k-1,\overline{\chi}}(0)^{-1}\sum_{a_1=1}^{n-1}a_1\sigma_{\ell-1,\chi}(a_1)\sigma_{k-1,\overline{\chi}}(n-a_1)\right|\\&\leq9.25\left(\frac{\pi en^2}{2(k-1)D}\right)^{k-1/2},
\end{align*}
which gives the result.   
\end{proof}
\begin{lemma}\label{lem:bdRprime}
       Let $3\leq\ell\leq\frac{K-4}{2}$ and $n\geq1$ be integers, $k=K-2-\ell$, $D\geq1$ an odd square-free integer, and $\chi$  a primitive Dirichlet character mod $D$. Then
    \begin{align*}
        |\mathcal{R}^{\prime}_{K,\ell,n,\chi}|\leq 18.5\left(\frac{\pi en^2}{(K-2)D}\right)^{\frac{K-1}{2}}.
    \end{align*}
\end{lemma}
\begin{proof}Again, by \eqref{eq:anotherboundforsigma}, we have
    \begin{align*}
        \left|\sum_{a_1=1}^{n-1}\sigma_{\ell-1,\chi}(a_1)(n-a_1)\sigma_{k-1,\overline{\chi}}(n-a_1)\right|&\leq\zeta(\ell-1)\zeta(k-1)\sum_{a_1=1}^{n-1}a_1^{\ell-1}(n-a_1)\cdot (n-a_1)^{k-1}\\&\leq\zeta(2)^2\sum_{ a_1=1}^{n-1}a_1^{k-3}(n-a_1)^{k-3}(n-a_1)^3\\&\leq\zeta(2)^2\left(\frac{n}{2}\right)^{2k-6}\sum_{a_1=1}^{n-1}(n-a_1)^3.
    \end{align*}
    Note that 
    $$\sum_{a_1=1}^{n-1}(n-a_1)^3=\frac{(n-1)^2n^2}{4}.$$
    Hence
        \begin{align*}
      \left|\sum_{a_1=1}^{n-1}\sigma_{\ell-1,\chi}(a_1)(n-a_1)\sigma_{k-1,\overline{\chi}}(n-a_1)\right|&\leq\zeta(2)^2\left(\frac{n}{2}\right)^{2k-6}\frac{n^4}{4}\leq4\zeta(2)^2\left(\frac{n}{2}\right)^{2k-1}
    \end{align*}
Using the bound \eqref{eq:bound2/L} again, we get   
\begin{align*}
   |\mathcal{R}_{K,\ell,n,\chi}^{\prime}|&=\left|-\frac{\ell}{k}\sigma_{k-1,\overline{\chi}}(0)^{-1}\sum_{a_1=1}^{n-1}\sigma_{\ell-1,\chi}(a_1)(n-a_1)\sigma_{k-1,\overline{\chi}}(n-a_1)\right|\\&\leq\frac{4\zeta(2)^2}{\sqrt{e}(2-\zeta(2))}\left(\frac{\pi en^2}{2(k-1)D}\right)^{k-1/2}
   \\&\leq18.5\left(\frac{\pi en^2}{(K-2)D}\right)^{\frac{K-1}{2}},
\end{align*}
    as desired.
\end{proof}
Next, we give the bound   of $\mathscr{E}_{K,\ell,n,\chi}$ in \eqref{eq:expressionoftheFoco}. 
\begin{lemma}\label{lem:bdDtermprod}
    Let $3\leq\ell\leq\frac{K-2}{2}$ and $n\geq2$ be integers, $k=K-\ell$, $D\geq1$ an odd square-free integer, and $\chi$ be a primitive Dirichlet character mod $D$. Then
     \begin{align*}
         |\mathscr{E}_{K,\ell,n,\chi}|\leq 16\left(\frac{\pi eDn^2}{K-2}\right)^{\frac{K-1}{2}}.
     \end{align*}
\end{lemma}
\begin{proof}
    Note that for all $\chi_2\neq\mathbbm{1}$ ($D_2\neq1$) and $m\geq0$, we have $|\sigma_{k-1,\chi_1,\chi_2}(m)|\leq \zeta(k-1)m^{k-1}$ by \eqref{eq:anotherboundforsigma}. Then 
    \begin{align}
        &\left|\sum_{\substack{D=D_1D_2\\D_2\neq 1 }}\overline{\chi}_2(-1)\sum_{\substack{a_1,a_2\geq0\\a_1+a_2=nD_2}}\sigma_{\ell-1,\chi_1,\overline{\chi}_{2}}(a_1)\sigma_{k-1,\overline{\chi}_{1},\chi_2}(a_2)\right|  \nonumber\\\leq&\zeta(\ell-1)\zeta(k-1)\sum_{\substack{D_2\mid D\\D_2>1}}\sum_{a_1=1}^{nD_2}a_1^{\ell-1}(nD_2-a_1)^{k-1}  \nonumber\\\leq&\zeta(\ell-1)\zeta(k-1)\sum_{\substack{D_2\mid D\\D_2>1}}\sum_{a_1=1}^{nD_2}a_1^{k-1}(nD_2-a_1)^{k-1}  \nonumber\\\leq&\zeta(\ell-1)\zeta(k-1)\sum_{\substack{D_2\mid D}}nD_2\left(\frac{nD_2}{2}\right)^{2k-2}\\\leq&2\zeta(\ell-1)\zeta(k-1)\left(\frac{n}{2}\right)^{2k-1}\sigma_{2k-1}(D)  \nonumber\\\leq&2\zeta(\ell-1)\zeta(k-1)\zeta(K-1)\left(\frac{nD}{2}\right)^{2k-1}\label{eq:accuratebdforcentralL}\\\leq&2\zeta(2)^3\left(\frac{nD}{2}\right)^{2k-1}.\label{eq:bdEpsilon}
    \end{align}
    Using the bound \eqref{eq:bound2/L} again, we get
    \begin{align}
\left|\mathscr{E}_{K,\ell,n,\chi}\right|=&\left|\sigma_{k-1,\overline{\chi}}(0)^{-1}\sum\limits_{\substack{D=D_1D_2\\D_2\neq 1}}\overline{\chi}_2(-1)\sum\limits_{\substack{a_1,a_2\geq0\\a_1+a_2=nD_2}}\sigma_{\ell-1,\chi_1,\overline{\chi}_{2}}(a_1)\sigma_{k-1,\overline{\chi}_{1},\chi_2}(a_2)\right|
        \nonumber  \\\leq&\frac{1}{\sqrt{e}(2-\zeta(2))}\left(\frac{2\pi e}{(k-1)D}\right)^{k-1/2}2\zeta(2)^3\left(\frac{nD}{2}\right)^{2k-1}  \nonumber\\=&\frac{2\zeta(2)^3}{\sqrt{e}(2-\zeta(2))}\left(\frac{\pi eDn^2}{2(k-1)}\right)^{k-1/2}  \nonumber\\\leq&16\left(\frac{\pi eDn^2}{K-2}\right)^{\frac{K-1}{2}},\label{eq:bdEpsilon2}
    \end{align}
    as desired.
\end{proof}
We show that $\mathscr{R}_{K,\ell,n,\chi}$ and $\mathscr{R}^{\prime}_{K,\ell,n,\chi}$  in \eqref{eq:fouriercoefficientsb} have similar bounds.
\begin{lemma}\label{lem:bdfrakR}
       Let $3\leq\ell\leq\frac{K-4}{2}$ and $n\geq1$ be integers, $k=K-2-\ell$, $D\geq1$ an odd square-free integer, and $\chi$ be a primitive Dirichlet character mod $D$. Then
    \begin{align*}
        |\mathscr{R}_{K,\ell,n,\chi}|\leq 16\left(\frac{\pi eDn^2}{K-2}\right)^{\frac{K-1}{2}}.
    \end{align*}
\end{lemma}
\begin{proof}Note that 
        \begin{align*}
        &\left|\sum\limits_{\substack{D=D_1D_2\\D_2\neq 1}}
\overline{\chi}_2(-1)D_2^{-1}\sum\limits_{\substack{a_1,a_2\geq0\\a_1+a_2=nD_2}}a_1\sigma_{\ell-1,\chi_1,\overline{\chi}_{2}}(a_1)\sigma_{k-1,\overline{\chi}_{1},\chi_2}(a_2)\right|\\\leq&\zeta(\ell-1)\zeta(k-1)\sum_{\substack{D_2\mid D\\D_2>1}}\sum_{a_1=1}^{n|D_2|}a_1\cdot a_1^{\ell-1}(nD_2-a_1)^{k-1}\\\leq&\zeta(2)^2\sum_{\substack{D_2\mid D\\D_2>1}}\sum_{a_1=1}^{n|D_2|}a_1^{k-1}(nD_2-a_1)^{k-1}\\
      \leq&2\zeta(2)^3\left(\frac{nD}{2}\right)^{2k-1},
    \end{align*}
as in \eqref{eq:bdEpsilon}. Thus, similar to \eqref{eq:bdEpsilon2}, we obtain
\begin{align*}|\mathscr{R}_{K,\ell,n,\chi}|
\leq16\left(\frac{\pi eDn^2}{K-2}\right)^{\frac{K-1}{2}},
\end{align*}
which gives the result.
\end{proof}
\begin{lemma}\label{lem:bdmathfrakRprime}
       Let $3\leq\ell\leq\frac{K-4}{2}$ and $n\geq1$  be integers, $k=K-2-\ell$, $D\geq1$ an odd square-free integer, and $\chi$ be a primitive Dirichlet character mod $D$. Then
    \begin{align}  |\mathscr{R}^{\prime}_{K,\ell,n,\chi}|\leq31\left(\frac{\pi eDn^2}{K-2}\right)^{\frac{K-1}{2}}.
    \end{align}
\end{lemma}
\begin{proof}Using \eqref{eq:anotherboundforsigma}, we see that
    \begin{align*}
       &\left|\sum\limits_{\substack{D=D_1D_2\\D_2\neq 1}}\overline{\chi}_2(-1)D_2^{-1}\sum\limits_{\substack{a_1,a_2\geq0\\a_1+a_2=nD_2}}\sigma_{\ell-1,\chi_1,\overline{\chi}_{2}}(a_1)a_2\sigma_{k-1,\overline{\chi}_{1},\chi_2}(a_2)\right|\\\leq&\zeta(\ell-1)\zeta(k-1)\sum_{\substack{D_2\mid D\\D_2>1}}\sum_{a_1=1}^{n|D_2|}a_1^{\ell-1}(nD_2-a_1)\cdot(nD_2-a_1)^{k-1}\\\leq&\zeta(2)^2\sum_{\substack{D_2\mid D\\D_2>1}}\sum_{a_1=1}^{n|D_2|}a_1^{k-3}(nD_2-a_1)^{k-3}(nD_2-a_1)^3\\\leq&\zeta(2)^2\sum_{\substack{D_2\mid D\\D_2>1}}\left(\frac{nD_2}{2}\right)^{2k-6}\sum_{a_1=1}^{n|D_2|}(nD_2-a_1)^3.
    \end{align*}
    Note that 
    \begin{align*}
      \sum_{a_1=1}^{n|D_2|}(nD_2-a_1)^3=\frac{(nD_2-1)^2(nD_2)^2}{4}\leq\frac{(nD_2)^4}{4}.
    \end{align*}
Hence
\begin{align*}
    &\left|\sum\limits_{\substack{D=D_1D_2\\D_2\neq 1}}\overline{\chi}_2(-1)D_2^{-1}\sum\limits_{\substack{a_1,a_2\geq0\\a_1+a_2=nD_2}}\sigma_{\ell-1,\chi_1,\overline{\chi}_{2}}(a_1)a_2\sigma_{k-1,\overline{\chi}_{1},\chi_2}(a_2)\right|\\\leq&4\zeta(2)^2\sum_{\substack{D_2\mid D\\D_2>1}}\left(\frac{nD_2}{2}\right)^{2k-2}\\\leq&4\zeta(2)^2\left(\frac{n}{2}\right)^{2k-2}\sigma_{2k-2}(D)\\\leq&4\zeta(2)^3\left(\frac{nD}{2}\right)^{2k-1},
\end{align*}
which implies that
    \begin{align*}
        |\mathscr{R}^{\prime}_{K,\ell,n,\chi}|&=\left|-\frac{\ell}{k}\cdot\sigma_{k-1,\overline{\chi}}(0)\sum\limits_{\substack{D=D_1D_2\\D_2\neq 1}}\overline{\chi}_2(-1)D_2^{-1}\sum\limits_{\substack{a_1,a_2\geq0\\a_1+a_2=nD_2}}\sigma_{\ell-1,\chi_1,\overline{\chi}_{2}}(a_1)a_2\sigma_{k-1,\overline{\chi}_{1},\chi_2}(a_2)\right|\\&\leq4\zeta(2)^3\left(\frac{nD}{2}\right)^{2k-1}\frac{1}{\sqrt{e}(2-\zeta(2))}\left(\frac{2\pi e}{(k-1)D}\right)^{k-1/2}\\&=\frac{4\zeta(2)^3}{\sqrt{e}(2-\zeta(2))}\left(\frac{\pi eDn^2}{K-2}\right)^{\frac{K-1}{2}}\\&\leq31\left(\frac{\pi eDn^2}{K-2}\right)^{\frac{K-1}{2}},
    \end{align*}
    as desired.
\end{proof}
Finally, combining the above results,  we give the asymptotics of $\mathfrak{a}_{K,\ell,\chi}(n)$ and $\mathfrak{b}_{K,\ell,\chi}(n)$.  The case $n=2^j$ suffices to serve our purpose even though a similar result holds true for general $n$. 
\begin{proposition}\label{prop:asymptoticA}
    Let $3\leq\ell\leq\frac{K-2}{2}$ and $j\geq0$  be integers. Then 
    \begin{align*}
        \mathfrak{a}_{K,\ell,\chi}(2^j)=\sigma_{\ell-1,\chi}(2^{j})(1+o(1)),
    \end{align*}
    where $o(1)\rightarrow 0$ when $K$ is sufficiently large relative to $j$ and D.
\end{proposition}
\begin{proof}
We know from \eqref{eq:expressionoftheFoco}  that
\begin{align}
     \frac{\mathfrak{a}_{K,\ell,\chi}(2^j)}{\sigma_{\ell-1,\chi}(2^j)}&=\frac{1+\frac{\sigma_{\ell-1,\chi}(0)}{\sigma_{k-1,\overline{\chi}}(0)}\frac{\sigma_{k-1,\overline{\chi}}(2^j)}{\sigma_{\ell-1,\chi}(2^j)}-\frac{2\sigma_{\ell-1,\chi}(0)}{\zeta(1-K) }\frac{\sigma_{K-1}(2^j)}{\sigma_{\ell-1,\chi}(2^j)}+\frac{\mathcal{E}_{K,\ell,2^j,\chi}}{\sigma_{\ell-1,\chi}(2^j)}+\frac{\mathscr{E}_{K,\ell,2^j,\chi}}{\sigma_{\ell-1,\chi}(2^j)}}{1+\frac{\sigma_{\ell-1,\chi}(0)}{\sigma_{k-1,\overline{\chi}}(0)}-\frac{2\sigma_{\ell-1,\chi}(0)}{\zeta(1-K) }+\mathscr{E}_{K,\ell,1,\chi}}.\label{eq:asymtoticFco}
\end{align}
It suffices to show that the non-constant terms in \eqref{eq:asymtoticFco} tend to zero when $K$ is sufficiently large relative to $j$ and $D$.
Note that
\begin{align}
    |\sigma_{k-1,\overline{\chi}}(2^j)|&=\left|\sum_{n=0}^j\overline{\chi}(2^n)2^{(k-1)n}\right|=\left|\frac{\overline{\chi}(2^{j+1})2^{(k-1)(j+1)}-1}{\overline{\chi}(2)2^{k-1}-1}\right|\leq\frac{2^{(k-1)(j+1)}+1}{2^{k-1}-1}\leq2\cdot 2^{(k-1)j},  \nonumber\\
    |\sigma_{\ell-1,\chi}(2^j)|&=\left|\frac{\chi(2^{j+1})2^{(\ell-1)(j+1)}-1}{\chi(2)2^{\ell-1}-1}\right|\geq\frac{2^{(\ell-1)(j+1)}-1}{2^{\ell-1}+1}\geq\frac{1}{2}2^{(\ell-1)j},\label{eq:lowerbddivisor}
\end{align}
which implies that 
\begin{align}
    \left|\frac{\sigma_{k-1,\chi}(2^j)}{\sigma_{\ell-1,\overline{\chi}}(2^j)} \right|\leq 4\cdot2^{(k-\ell)j}.\label{eq:quotientofdivisor}
\end{align}
By Lemma \ref{lem:bdquotientLvalues} and \eqref{eq:quotientofdivisor}, we have 
\begin{align}
    \left|\frac{\sigma_{\ell-1,\chi}(0)}{\sigma_{k-1,\overline{\chi}}(0)}\frac{\sigma_{k-1,\overline{\chi}}(2^j)}{\sigma_{\ell-1,\chi}(2^j)}\right|
    \leq& 6\left(\frac{\ell-1}{k-1}\right)^{\ell-1/2}\left(\frac{2\pi e}{(k-1)D}\right)^{k-\ell}4\cdot2^{(k-\ell)j}  \nonumber\\
    \leq&24\left(\frac{\pi e 2^{j+1}}{(k-1)D}\right)^{k-\ell}\label{eq:bdagain}\\=&o(1).  \nonumber
\end{align}
Lemma \ref{lem:bdquotientLandzeta} together with \eqref{eq:quotientofdivisor} (with $k$ replaced by $K$) gives
\begin{align}
    \left|\frac{2\sigma_{\ell-1,\chi}(0)}{\zeta(1-K) }\frac{\sigma_{K-1}(2^j)}{\sigma_{\ell-1,\chi}(2^j)}\right|&\leq 2\left(\frac{(\ell-1)D}{K-1}\right)^{\ell-1/2}\left(\frac{2\pi e}{K-1}\right)^{K-\ell}4\cdot 2^{(K-\ell)j}  \nonumber\\&=8\left(\frac{(\ell-1)D}{K-1}\right)^{\ell-1/2}\left(\frac{2\pi e\cdot 2^j}{K-1}\right)^{K-\ell}\label{eq:bdquotienttwistandzeta}
    \\&=o(1).  \nonumber
\end{align}
By Lemma \ref{lem:bdsmalltermprod} and \eqref{eq:lowerbddivisor}, we have 
\begin{align*}
    \left|\frac{\mathcal{E}_{K,\ell,2^j,\chi}}{\sigma_{\ell-1,\chi}(2^j)}\right|&\leq 9.25\left(\frac{\pi e2^{2j}}{(K-2)D}\right)^{\frac{K-1}{2}}2\cdot 2^{-(\ell-1)j}=o(1).
\end{align*}
Lemma \ref{lem:bdDtermprod} together with \eqref{eq:lowerbddivisor} implies that
\begin{align*}
  \left|  \frac{\mathscr{E}_{K,\ell,2^j,\chi}}{\sigma_{\ell-1,\chi}(2^j)}\right|&\leq16\left(\frac{\pi eD4^j}{K-2}\right)^{\frac{K-1}{2}}2\cdot2^{-(\ell-1)j}=o(1).
\end{align*}
By Lemmas \ref{lem:bdquotientLvalues}, \ref{lem:bdquotientLandzeta} and \ref{lem:bdDtermprod}, we have  
\begin{align*}
    \left|\frac{\sigma_{\ell-1,\chi}(0)}{\sigma_{k-1,\overline{\chi}}(0)}\right|=o(1),\quad\left|\frac{2\sigma_{\ell-1,\chi}(0)}{\zeta(1-K)}\right|=o(1),\quad{\rm and}\quad|\mathscr{E}_{K,\ell,1,\chi}|=o(1),
\end{align*}
which completes the proof.
\end{proof}

\begin{proposition}\label{prop:asymptoticB}
   Let $3\leq\ell\leq\frac{K-4}{2}$ and $j\geq0$ be integers. Then
    \begin{align*}
        \mathfrak{b}_{K,\ell,\chi}(2^j)=2^{j}\sigma_{\ell-1,\chi}(2^{j})(1+o(1)),
    \end{align*}
    where $o(1)\rightarrow 0$ when $K$ is sufficiently large relative to $j$ and D.
\end{proposition}
\begin{proof}
We know from \eqref{eq:fouriercoefficientsb} that $    \frac{\mathfrak{b}_{K,\ell,\chi}(2^j)}{2^{j}\sigma_{\ell-1,\chi}(2^{j})}=$
\begin{align}
\frac{1-\frac{\sigma_{\ell-1,\chi}(0)}{\sigma_{k-1,\overline{\chi}}(0)}\frac{\ell}{k}\frac{\sigma_{k-1,\overline{\chi}}(2^j)}{\sigma_{\ell-1,\chi}(2^{j})}+\frac{\mathcal{R}_{K,\ell,2^j,\chi}}{2^{j}\sigma_{\ell-1,\chi}(2^{j})}+\frac{\mathcal{R}_{K,\ell,2^j,\chi}^{\prime}}{2^j\sigma_{\ell-1,\chi}(2^{j})}+\frac{\mathscr{R}_{K,\ell,2^j,\chi}}{2^j\sigma_{\ell-1,\chi}(2^{j})}+\frac{\mathscr{R}_{K,\ell,2^j,\chi}^{\prime}}{2^j\sigma_{\ell-1,\chi}(2^{j})}}{1-\frac{\sigma_{\ell-1,\chi}(0)}{\sigma_{k-1,\overline{\chi}}(0)}\frac{\ell}{k}+\mathscr{R}_{K,\ell,1,\chi}+\mathscr{R}^{\prime}_{K,\ell,1,\chi}}.\label{eq:asymptoticsGco}
\end{align}
It suffices to show that the non-constant terms in \eqref{eq:asymptoticsGco} tend to zero when $K$ is sufficiently large relative to $j$ and $D$. Note that \eqref{eq:bdagain} implies that
\begin{align*}
    \left|\frac{\sigma_{\ell-1,\chi}(0)}{\sigma_{k-1,\overline{\chi}}(0)}\frac{\ell}{k}\frac{\sigma_{k-1,\overline{\chi}}(2^j)}{\sigma_{\ell-1,\chi}(2^j)}\right|=o(1).
\end{align*}
By Lemma \ref{lem:bdR} and \eqref{eq:lowerbddivisor},
\begin{align}
    \left|\frac{\mathcal{R}_{K,\ell,2^j,\chi}}{2^{j}\sigma_{\ell-1,\chi}(2^{j})}\right|&\leq 9.25\left(\frac{\pi e4^j}{(K-2)D}\right)^{\frac{K-1}{2}}2^{-j}\cdot 2\cdot2^{-(\ell-1)j}=o(1).
\end{align}
By Lemma \ref{lem:bdRprime} and \eqref{eq:lowerbddivisor}, 
\begin{align*}
    \left|\frac{\mathcal{R}^{\prime}_{K,\ell,2^j,\chi}}{2^{j}\sigma_{\ell-1,\chi}(2^{j})}\right|\leq18.5\left(\frac{\pi e4^j}{(K-2)D}\right)^{\frac{K-1}{2}}2^{-j}\cdot 2\cdot2^{-(\ell-1)j}=o(1).
\end{align*}
By Lemma \ref{lem:bdfrakR} and \eqref{eq:lowerbddivisor}, we get
\begin{align*}
  \left|  \frac{\mathscr{R}_{K,\ell,2^j,\chi}}{2^j\sigma_{\ell-1,\chi}(2^{j})}\right|&\leq16\left(\frac{\pi eD4^j}{K-2}\right)^{\frac{K-1}{2}}2^{-j}\cdot 2\cdot 2^{-(\ell-1)j}=o(1).
\end{align*}
By Lemma \ref{lem:bdmathfrakRprime} and \eqref{eq:lowerbddivisor}, we have 
\begin{align*}
      \left|  \frac{\mathscr{R}_{K,\ell,2^j,\chi}^{\prime}}{2^j\sigma_{\ell-1,\chi}(2^{j})}\right|&\leq31\left(\frac{\pi eD4^j}{K-2}\right)^{\frac{K-1}{2}}2^{-j}\cdot 2\cdot 2^{-(\ell-1)j}=o(1).
\end{align*}
By Lemmas \ref{lem:bdquotientLvalues},  \ref{lem:bdfrakR} and \ref{lem:bdmathfrakRprime}, we get 
\begin{align}
   \left|\frac{\sigma_{\ell-1,\chi}(0)}{\sigma_{k-1,\overline{\chi}}(0)}\frac{\ell}{k}\right|=o(1),\quad\left|\mathscr{R}_{K,\ell,1,\chi}\right|=o(1),\quad{\rm and}\quad\left|\mathscr{R}^{\prime}_{K,\ell,1,\chi}\right|=o(1). 
\end{align}
This completes the proof.
\end{proof}
\begin{remark}
    From the proofs above, we can see that Propositions \ref{prop:asymptoticA} and \ref{prop:asymptoticB} still hold if we replace $2^j$ by $p^j$, where $o(1)\rightarrow0$ when $K$ is sufficiently large relative to $j,p$ and $D$.
\end{remark}
\section{Coefficient matrices}\label{sect:coefmatrix}
In this section, we define matrices that are used to study the linear independence of the corresponding twisted periods. More precisely, we prove Theorems \ref{thm:mainthmfixDdifferentell}, \ref{thm:mainthmfixDdifferentell2}, \ref{thm:twistbydifferentD} and \ref{thm:twistbydifferentD2} by showing the  matrices, formed by the (normalized) $q^{2^i}$-Fourier coefficients of relevant kernel functions for $0\le i\le n-1$, are non-singular. 

We start with the proof of Theorem \ref{thm:mainthmfixDdifferentell}. For $1\leq i\leq n$, we pick out the Fourier coefficients $\mathfrak{a}_{K,\ell_i,\chi}(1)$, $\mathfrak{a}_{K,\ell_i,\chi}(2)$, $\mathfrak{a}_{K,\ell_i,\chi}(2^2)$, $...$, $\mathfrak{a}_{K,\ell_i,\chi}(2^{n-1})$ of each $\mathfrak{F}_{K,\ell_i,\chi}$ \eqref{eq:F-normalized}  and place them in a $(n,n)$-matrix
\begin{align}
    M_{K,\ell_1,...,\ell_n,\chi}:=\left[a_{K,\ell_i,\chi}(2^{j-1})\right]_{1\leq i,j\leq n}=\left[\sigma_{\ell_i-1,\chi}(2^{j-1})(1+\epsilon_{K,i,j})\right]_{1\leq i,j\leq n}.\label{eq:matrix1}
\end{align}
To show that $\{\mathfrak{F}_{K,\ell_i,\chi}\}_{i=1}^n$ is linearly independent, it suffices to show that $M_{K,\ell_1,...,\ell_n}$ is non-singular. Such a choice of $M_{K,\ell_1,...,\ell_n,\chi}$ allows us to relate it to the well-known  Vandermonde  matrix in terms of determinant. More specifically, we  ignore the error terms $\epsilon_{K,i,j}$ ($1\leq i,j\leq n$) in  $M_{K,\ell_1,...,\ell_n,\chi}$ and define the following matrix
\begin{align}
    \mathbf{M}_{K,\ell_1,...,\ell_n,\chi}:&=\begin{bmatrix}
        1 & \sigma_{\ell_1-1,\chi}(2) & \sigma_{\ell_1-1,\chi}(2^2) & \cdots & \sigma_{\ell_1-1,\chi}(2^{n-1})\\
        1 &  \sigma_{\ell_2-1,\chi}(2) &  \sigma_{\ell_2-1,\chi}(2^2) & \cdots &  \sigma_{\ell_2-1,\chi}(2^{n-1})\\
        \vdots & \vdots & \vdots & \ddots & \vdots\\
        1 & \sigma_{\ell_{n-1}-1,\chi}(2) & \sigma_{\ell_{n-1}-1,\chi}(2^2) & \cdots & \sigma_{\ell_{n-1}-1,\chi}(2^{n-1})\\ 
        1 & \sigma_{\ell_n-1,\chi}(2) & \sigma_{\ell_n-1,\chi}(2^2) & \cdots & \sigma_{\ell_n-1,\chi}(2^{n-1})
    \end{bmatrix}.\label{eq:limitmatrix}
\end{align}
We have the following lower bound for $\det \mathbf{M}_{K,\ell_1,...,\ell_n,\chi}$.
\begin{lemma}\label{lem:lowerboundoflimitmatrix}
   Let $\mathbf{M}_{K,\ell_1,...,\ell_n,\chi}$ be defined as in \eqref{eq:limitmatrix}. Then 
   \begin{align*}
\left|\det \mathbf{M}_{K,\ell_1,...,\ell_n,\chi}\right|\geq\left(\frac{3}{4}\right)^{\frac{n(n-1)}{2}}\prod_{i=2}^n2^{(i-1)(\ell_i-1)}.
   \end{align*}
\end{lemma}
\begin{proof}
    After column reductions, we have 
\begin{align*}
\left|\det \mathbf{M}_{K,\ell_1,...,\ell_n,\chi}\right|&=\left|\det\begin{bmatrix}
        1 & \chi(2)2^{\ell_1-1}  & \cdots & (\chi(2)2^{\ell_1-1})^{n-1}\\
        \vdots  & \vdots & \ddots & \vdots\\ 
       1 & \chi(2)2^{\ell_{n}-1} &  \cdots & (\chi(2)2^{\ell_{n}-1})^{n-1}
    \end{bmatrix}\right|\\&=\left|\prod_{1\leq i<j\leq n}\left(\chi(2)2^{\ell_j-1}-\chi(2)2^{\ell_i-1}\right)\right|\\&=\prod_{1\leq i<j\leq n}2^{\ell_j-1}\cdot(1-2^{\ell_i-\ell_j})\\&\geq\prod_{1\leq i<j\leq n}\frac{3}{4}\cdot2^{\ell_j-1}\\&=\left(\frac{3}{4}\right)^{\frac{n(n-1)}{2}}\prod_{i=2}^n2^{(i-1)(\ell_i-1)},
\end{align*}
where the second equality is from the determinant of the Vandermonde matrix.
\end{proof}
\begin{lemma}
    For any $\tau\in S_n$, we have 
    \begin{align*}
        \left|\prod_{i=1}^n\sigma_{\ell_i-1,\chi}(2^{\tau(i)-1})\right|<2^{n-1}\prod_{i=2}^n2^{(i-1)(\ell_i-1)}.
    \end{align*}
\end{lemma}
\begin{proof}
    Note that
    \begin{align*}
        \left|\prod_{i=1}^n\sigma_{\ell_i-1,\chi}(2^{\tau(i)-1})\right|\leq\prod_{i=1}^n\sigma_{\ell_i-1}(2^{\tau(i)-1}).
    \end{align*}
    By \cite[Corollary 2.5]{oddperiods}, we know that 
   $\prod_{i=1}^n\sigma_{\ell_i-1}(2^{\tau(i)-1})$ attains maximum when $\tau$ is the identity permutation. It follows that
   \begin{align}
        \left|\prod_{i=1}^n\sigma_{\ell_i-1,\chi}(2^{\tau(i)-1})\right|\leq\prod_{i=1}^n\sigma_{\ell_i-1}(2^{i-1})<\prod_{i=1}^n2\cdot2^{(i-1)(\ell_i-1)}=2^{n-1}\prod_{i=2}^n2^{(i-1)(\ell_i-1)},\label{eq:midstep}
        \end{align}
        which finishes the proof.
\end{proof}
We now give a criterion for the nonsingularity of $M_{K,\ell_1,...,\ell_n,\chi}$, see also \cite[Lemma 2.7]{oddperiods}.
\begin{proposition}\label{prop:criterionM}
    Define a function
    \begin{align}
f(x_1,...,x_n):=\left(\prod_{i=1}^n(1+x_i)\right)-1.\label{eq:functionf}
    \end{align}
  Suppose that 
    \begin{align}
        \sup_{\tau\in S_n}|f(\epsilon_{K,1,\tau(1)},...,\epsilon_{K,n,\tau(n)})|<n!^{-1}\cdot\left(\frac{3}{4}\right)^{\frac{n(n-1)}{2}}\cdot2^{1-n}.\label{eq:criterionM}
    \end{align}
    Then $M_{K,\ell_1,...,\ell_n,\chi}$ is nonsingular.
\end{proposition}
\begin{proof}
Note that
    \begin{align*}
        \det M_{K,\ell_1,...,\ell_n,\chi}&=\sum_{\tau\in S_n}(-1)^{\sign(\tau)}\prod_{i=1}^n\sigma_{\ell_i-1,\chi}(2^{\tau(i)-1})(1+\epsilon_{K,i,\tau(i)})\\&=\det \mathbf{M}_{K,\ell_1,...,\ell_n,\chi}+\sum_{\tau\in S_n}(-1)^{\sign(\tau)}f(\epsilon_{K,1,\tau(1)},...,\epsilon_{K,n,\tau(n)})\prod_{i=1}^n\sigma_{\ell_i-1,\chi}(2^{\tau(i)-1}).
    \end{align*}
  On the one hand, we have from Lemma \ref{lem:lowerboundoflimitmatrix} that  
\begin{align*}
\left|\det \mathbf{M}_{K,\ell_1,...,\ell_n,\chi}\right|\geq\left(\frac{3}{4}\right)^{\frac{n(n-1)}{2}}\prod_{i=2}^n2^{(i-1)(\ell_i-1)}.
   \end{align*}
   On the other hand, we have
   \begin{align*}
       &\left|\sum_{\tau\in S_n}(-1)^{\sign(\tau)}f(\epsilon_{K,1,\tau(1)},...,\epsilon_{K,n,\tau(n)})\prod_{i=1}^n\sigma_{\ell_i-1,\chi}(2^{\tau(i)-1})\right|\\\leq & n! \sup_{\tau\in S_n}|f(\epsilon_{K,1,\tau(1)},...,\epsilon_{K,n,\tau(n)})|\left|\prod_{i=1}^n\sigma_{\ell_i-1,\chi}(2^{\tau(i)-1}) \right|\\ <&\left(\frac{3}{4}\right)^{\frac{n(n-1)}{2}}\prod_{i=2}^n2^{(i-1)(\ell_i-1)}.
   \end{align*}
   Therefore,  $\det M_{K,\ell_1,...,\ell_n}\neq 0$.
\end{proof}
\begin{lemma}\label{lem:inequalityforf}
    Let $f$ be defined as in \eqref{eq:functionf}. Suppose there is an $M>0$  such that $|x_i|\leq M$ for all $1\leq i\leq n$. Then 
    \begin{align*}
        |f(x_1,...,x_n)|\leq(1+M)^{n}-1.
    \end{align*}
\end{lemma}
\begin{proof}
    Note that 
    \begin{align}
|f(x_1,...,x_n)|&=\left|\prod_{1\leq i\leq n}(1+x_i)-1\right|  \nonumber\\&=
\left|\sum_{i=1}^nx_{i}+\sum_{1\leq i<j\leq n}x_ix_j+\cdots+x_1x_2\cdots x_n\right|  \nonumber\\&\leq\binom{n}{1}M+\binom{n}{2}M^2+\cdots+\binom{n}{n}M^n  \nonumber\\&=(1+M)^n-1,\label{eq:boundoff}\end{align}
as desired.
\end{proof}
Theorem \ref{thm:mainthmfixDdifferentell} follows from the proposition below.
\begin{proposition}\label{prop:tpyeinonsingular}
   Let $n\geq1$ be an integer, $D\geq1$ an odd square-free integer, and $\chi$ be a primitive Dirichlet character mod $D$. For $K\gg_{n,D}1$, if $3\leq\ell_1<\ell_2<...<\ell_n\leq\frac{K-2}{2}$ are integers such that $\chi(-1)=(-1)^{\ell_i}$ for all $1\leq i\leq n$, then $M_{K,\ell_1,...,\ell_n,\chi}$ is non-singular.
\end{proposition}
\begin{proof} 
From Proposition \ref{prop:asymptoticA} we know that $\epsilon_{K,i,\tau(i)}=o(1)$ for all $1\leq i\leq n$ and $\tau\in S_n$, where $o(1)\rightarrow0$ when $K$ is sufficiently large relative to $n$ and $D$. Then Lemma \ref{lem:inequalityforf} implies that
    \begin{align*}
          \sup_{\tau\in S_n}|f(\epsilon_{K,1,\tau(1)},...,\epsilon_{K,n,\tau(n)})|&\leq(1+o(1))^n-1=o(1).
    \end{align*}
In particular, $K$ can be chosen such that \eqref{eq:criterionM} is satisfied,
 which shows that $M_{K,\ell_1,...,\ell_n,\chi}$ is non-singular for $K\gg_{n,D}1$.
\end{proof}
\begin{remark}
     It is possible to obtain an explicit lower bound of $K$ in terms of $n$ and $D$ for Proposition \ref{prop:tpyeinonsingular}. However, it will be too messy. Thus, we omit it. 
\end{remark}
We now consider the matrix that is used to prove Theorem \ref{thm:mainthmfixDdifferentell2}. For $1\leq i\leq n$ we pick out the Fourier coefficients $\mathfrak{b}_{K,\ell_i,\chi}(1), \mathfrak{b}_{K,\ell_i,\chi}(2), \mathfrak{b}_{K,\ell_i,\chi}(2^2),...,\mathfrak{b}_{K,\ell_i,\chi}(2^{n-1})$ of each $\mathfrak{G}_{K,\ell_i,\chi}$ \eqref{eq:G-normalized} and place them in a $(n,n)$-matrix
\begin{align}
    N_{K,\ell_1,...,\ell_n,\chi}:=\left[\mathfrak{b}_{K,\ell_i,\chi}(2^{j-1})\right]_{1\leq i,j\leq n}=\left[2^{j-1}\sigma_{\ell_i-1,\chi}(2^{j-1})(1+\delta_{K,i,j})\right]_{1\leq i,j\leq n}.\label{eq:matrix2}
\end{align}
To prove that $\{\mathfrak{G}_{K,\ell_i,\chi}\}_{i=1}^n$ is linearly independent, it suffices to show that $N_{K,\ell_1,...,\ell_n,\chi}$ is non-singular. Define the matrix
\begin{align}
\mathbf{N}_{K,\ell_1,...,\ell_n,\chi}:&=\begin{bmatrix}
        1 & 2\sigma_{\ell_1-1,\chi}(2) & 2^2\sigma_{\ell_1-1,\chi}(2^2) & \cdots & 2^{n-1}\sigma_{\ell_1-1,\chi}(2^{n-1})\\
        1 &  2\sigma_{\ell_2-1,\chi}(2) &  2^2\sigma_{\ell_2-1,\chi}(2^2) & \cdots &  2^{n-1}\sigma_{\ell_2-1,\chi}(2^{n-1})\\
        \vdots & \vdots & \vdots & \ddots & \vdots\\
        1 & 2\sigma_{\ell_{n-1}-1,\chi}(2) & 2^2\sigma_{\ell_{n-1}-1,\chi}(2^2) & \cdots & 2^{n-1}\sigma_{\ell_{n-1}-1,\chi}(2^{n-1})\\ 
        1 & 2\sigma_{\ell_n-1,\chi}(2) & 2^2\sigma_{\ell_n-1,\chi}(2^2) & \cdots & 2^{n-1}\sigma_{\ell_n-1,\chi}(2^{n-1})
    \end{bmatrix}.\label{eq:limitmatrixtype2}
\end{align}
\begin{lemma}\label{lem:lowerboundoflimitmatrix2}
   Let $\mathbf{N}_{K,\ell_1,...,\ell_n,\chi}$ be defined as in \eqref{eq:limitmatrixtype2}. Then 
   \begin{align*}
\left|\det \mathbf{N}_{K,\ell_1,...,\ell_n,\chi}\right|\geq\left(\frac{3}{2}\right)^{\frac{n(n-1)}{2}}\prod_{i=2}^n2^{(i-1)(\ell_i-1)}.
   \end{align*}
\end{lemma}
\begin{proof}
    Since $
\det\mathbf{N}_{K,\ell_1,...,\ell_n,\chi}=2^{\frac{n(n-1)}{2}}\det\mathbf{M}_{K,\ell_1,...,\ell_n,\chi},
$
the result follows from Lemma \ref{lem:lowerboundoflimitmatrix}.
\end{proof}
\begin{lemma}\label{lem:permutationRankin}
    Let $\tau\in S_n$ be a permutation of the set $\{1,2,...,n\}$. Then
    \begin{align*}
        \prod_{i=1}^n2^{\tau(i)-1}\sigma_{\ell_i-1}(2^{\tau(i)-1})
    \end{align*}
    reaches its maximum exactly when $\tau$ is the identity permutation.
\end{lemma}
\begin{proof}
 Note that   \cite[Lemma 2.4]{oddperiods} asserts that for $\ell_2>\ell_1\geq2$ and $n_2>n_1\geq0$, we have
 \begin{align*}
     \sigma_{\ell_2-1}(2^{n_2})\cdot\sigma_{\ell_1-1}(2^{n_1})\geq  \sigma_{\ell_2-1}(2^{n_1})\cdot\sigma_{\ell_1-1}(2^{n_2}),
 \end{align*}
which implies that
\begin{align*}
         2^{n_2}\sigma_{\ell_2-1}(2^{n_2})\cdot2^{n_1}\sigma_{\ell_1-1}(2^{n_1})\geq  2^{n_1}\sigma_{\ell_2-1}(2^{n_1})\cdot2^{n_2}\sigma_{\ell_1-1}(2^{n_2}).
\end{align*}
So the result follows.
\end{proof}
\begin{lemma}
    For any $\tau\in S_n$, we have 
    \begin{align*}
        \left|\prod_{i=1}^n2^{\tau(i)-1}\sigma_{\ell_i-1,\chi}(2^{\tau(i)-1})\right|<2^{\frac{n(n-1)}{2}}\cdot2^{n-1}\prod_{i=2}^n2^{(i-1)(\ell_i-1)}.
    \end{align*}
\end{lemma}
\begin{proof}
    By Lemma \ref{lem:permutationRankin}, we have
   \begin{align*}
        \left|\prod_{i=1}^n2^{\tau(i)-1}\sigma_{\ell_i-1,\chi}(2^{\tau(i)-1})\right|&\leq \prod_{i=1}^n2^{\tau(i)-1}\sigma_{\ell_i-1}(2^{\tau(i)-1})\\&\leq\prod_{i=1}^n2^{i-1}\sigma_{\ell_i-1}(2^{i-1})\\&=2^{\frac{n(n-1)}{2}}\prod_{i=1}^n\sigma_{\ell_i-1}(2^{i-1})\\&\leq2^{\frac{n(n-1)}{2}}\cdot2^{n-1}\prod_{i=2}^n2^{(i-1)(\ell_i-1)},
        \end{align*}
        where the last inequality comes from \eqref{eq:midstep}.
\end{proof}

\begin{proposition}
    Let $f$ be defined as in \eqref{eq:functionf}.  Suppose that 
    \begin{align}
        \sup_{\tau\in S_n}|f(\delta_{K,1,\tau(1)},...,\delta_{K,n,\tau(n)})|<n!^{-1}\cdot \left(\frac{3}{4}\right)^{\frac{n(n-1)}{2}}\cdot2^{1-n}.\label{eq:criterion2}
    \end{align}
    Then $N_{K,\ell_1,...,\ell_n,\chi}$ is nonsingular.
\end{proposition}
\begin{proof}
Note that $\det N_{K,\ell_1,...,\ell_n,\chi}=$
    \begin{align*}
        &\sum_{\tau\in S_n}(-1)^{\sign(\tau)}\prod_{i=1}^n2^{\tau(i)-1}\sigma_{\ell_i-1,\chi}(2^{\tau(i)-1})(1+\delta_{K,i,\tau(i)})\\=&\det \mathbf{N}_{K,\ell_1,...,\ell_n,\chi}+\sum_{\tau\in S_n}(-1)^{\sign(\tau)}f(\delta_{K,1,\tau(1)},...,\delta_{K,n,\tau(n)})\prod_{i=1}^n2^{\tau(i)-1}\sigma_{\ell_i-1,\chi}(2^{\tau(i)-1}).
    \end{align*}
  We have from Lemma \ref{lem:lowerboundoflimitmatrix2} that  
\begin{align*}
\left|\det \mathbf{N}_{K,\ell_1,...,\ell_n,\chi}\right|\geq\left(\frac{3}{2}\right)^{\frac{n(n-1)}{2}}\prod_{i=2}^n2^{(i-1)(\ell_i-1)}.
   \end{align*}
 Note  also that
   \begin{align*}
       &\left|\sum_{\tau\in S_n}(-1)^{\sign(\tau)}f(\delta_{K,1,\tau(1)},...,\delta_{K,n,\tau(n)})\prod_{i=1}^n2^{\tau(i)-1}\sigma_{\ell_i-1,\chi}(2^{\tau(i)-1})\right|\\\leq & n! \sup_{\tau\in S_n}|f(\delta_{K,1,\tau(1)},...,\delta_{K,n,\tau(n)})|\left|\prod_{i=1}^n2^{\tau(i)-1}\sigma_{\ell_i-1,\chi}(2^{\tau(i)-1}) \right|\\ <&\left(\frac{3}{2}\right)^{\frac{n(n-1)}{2}}\prod_{i=2}^n2^{(i-1)(\ell_i-1)}.
   \end{align*}
   Therefore,  $\det N_{K,\ell_1,...,\ell_n,\chi}\neq 0$.
\end{proof}
Theorem \ref{thm:mainthmfixDdifferentell2} follows from the result below.
\begin{proposition}\label{prop:typeiinonsingular}
    Let $n\geq1$ be an integer, $D\geq1$ an odd square-free integer, and $\chi$ be a primitive Dirichlet character mod $D$. For $K\gg_{n,D}1$, if $3\leq\ell_1<\ell_2<...<\ell_n\leq\frac{K-4}{2}$ are integers such that $\chi(-1)=(-1)^{\ell_i}$ for all $1\leq i\leq n$, then $N_{K,\ell_1,...,\ell_n,\chi}$ is non-singular.
\end{proposition}
\begin{proof}
    By Proposition \ref{prop:asymptoticB}, we have $\delta_{K,i,\tau(i)}=o(1)$ for all $1\leq i\leq n$ and $\tau\in S_n$,    where $o(1)\rightarrow0$ when $K$ is sufficiently large relative to $n$ and $D$. Thus, Lemma \ref{lem:inequalityforf} implies that 
    \begin{align*}
          \sup_{\tau\in S_n}|f(\delta_{K,1,\tau(1)},...,\delta_{K,n,\tau(n)})|&\leq(1+o(1))^n-1=o(1),
    \end{align*}
  In particular, $K$ can be chosen such that \eqref{eq:criterion2} is satisfied, showing that $N_{K,\ell_1,...,\ell_n,\chi}$ is non-singular for $K\gg_{n,D}1$.
\end{proof}
Next, we define the matrix that is used to prove Theorem \ref{thm:twistbydifferentD}.
Let $\chi_i$ be primitive Dirichlet characters mod $D$ such that $\chi_i(-1)=(-1)^{\ell}$  and $\chi_i(2)$ are pairwise distinct for all $1\leq i\leq n$. Choose the Fourier coefficients $\mathfrak{a}_{K,\ell,\chi_1}(1), \mathfrak{a}_{K,\ell,\chi_2}(2),...,\mathfrak{a}_{K,\ell,\chi_n}(2^{n-1})$ of each $\mathfrak{F}_{K,\ell,\chi_i}$ ($1\leq i\leq n$) and place them in a $(n,n)$-matrix
\begin{align}
    P_{K,\ell, \chi_1,...,\chi_n}:=\left[\mathfrak{a}_{K,\ell,\chi_i}(2^{j-1})\right]_{1\leq i,j\leq n}=\left[\sigma_{\ell-1,\chi_i}(2^{j-1})(1+\eta_{K,i,j})\right]_{1\leq i,j\leq n}.\label{eq:matrix3}
\end{align}
Define the following matrix
\begin{align}
\mathbf{P}_{K,\ell,\chi_1,...,\chi_n}:&=\begin{bmatrix}
        1 & \sigma_{\ell-1,\chi_1}(2) & \sigma_{\ell-1,\chi_1}(2^2) & \cdots & \sigma_{\ell-1,\chi_1}(2^{n-1})\\
        1 &  \sigma_{\ell-1,\chi_2}(2) &  \sigma_{\ell-1,\chi_2}(2^2) & \cdots &  \sigma_{\ell-1,\chi_2}(2^{n-1})\\
        \vdots & \vdots & \vdots & \ddots & \vdots\\
        1 & \sigma_{\ell-1,\chi_{n-1}}(2) & \sigma_{\ell-1,\chi_{n-1}}(2^2) & \cdots & \sigma_{\ell-1,\chi_{n-1}}(2^{n-1})\\ 
        1 & \sigma_{\ell-1,\chi_n}(2) & \sigma_{\ell-1,\chi_n}(2^2) & \cdots & \sigma_{\ell-1,\chi_n}(2^{n-1})
    \end{bmatrix}.\label{eq:limitmatrixtype3}
\end{align}
\begin{lemma}\label{lem:lowerboundoflimitmatrix3}
   Let $\mathbf{P}_{K,\ell,\chi_1,...,\chi_n}$ be defined as in \eqref{eq:limitmatrixtype3}. Then 
   \begin{align*}
\left|\det \mathbf{P}_{K,\ell,\chi_1,...,\chi_n}\right|\geq\left(\frac{2^{\ell+1}}{D}\right)^{\frac{n(n-1)}{2}}.
   \end{align*}
\end{lemma}
\begin{proof}
        After column reductions, we have 
\begin{align*}
\left|\det \mathbf{P}_{K,\ell,\chi_1,...,\chi_n}\right|&=\left|\det\begin{bmatrix}
        1 & \chi_1(2)2^{\ell-1}  & \cdots & (\chi_1(2)2^{\ell-1})^{n-1}\\
        \vdots  & \vdots & \ddots & \vdots\\ 
       1 & \chi_n(2)2^{\ell-1} &  \cdots & (\chi_n(2)2^{\ell-1})^{n-1}
    \end{bmatrix}\right|\\&=\prod_{1\leq i<j\leq n}2^{\ell-1}\left|\chi_j(2)-\chi_i(2)\right|
\end{align*}
by the determinant of the Vandermonde matrix. Since $\chi_i$ and $\chi_j$ are primitive Dirichlet characters mod $D$ such that $\chi_i(2)\neq\chi_j(2)$, we may assume that $\chi_j(2)=e^{\frac{2\pi i}{\varphi(D)}r}$ and $\chi_j(2)=e^{\frac{2\pi i}{\varphi(D)}s}$, where $0\leq r<s<\varphi(D)$ and $\varphi$ is the Euler-totient function. Note that 
\begin{align*}
    |\chi_j(2)-\chi_i(2)|&=\left|e^{\frac{2\pi i}{\varphi(D)}r}-e^{\frac{2\pi i}{\varphi(D)}s}\right|\\&\geq\left|e^{\frac{2\pi i}{\varphi(D)}}-1\right|\\&
    \geq2\sin\left(\frac{1}{2}\cdot\frac{2\pi}{\varphi(D)}\right)\\&\geq2\cdot\frac{2}{\pi}\cdot\frac{\pi}{\varphi(D)}\\&\geq\frac{4}{D}.
\end{align*}
It follows that 
\begin{align*}
    \left|\det \mathbf{P}_{K,\ell,\chi_1,...,\chi_n}\right|&\geq\prod_{i=1}^n 2^{\ell-1}4D^{-1}=\left(\frac{2^{\ell+1}}{D}\right)^{\frac{n(n-1)}{2}},
\end{align*}
as desired.
\end{proof}
\begin{proposition}
    Let $f$ be defined as in \eqref{eq:functionf}.  Suppose that 
    \begin{align}
        \sup_{\tau\in S_n}|f(\eta_{K,1,\tau(1)},...,\eta_{K,n,\tau(n)})|<n!^{-1}\cdot2^{1-n}\cdot\left(\frac{4}{D}\right)^{\frac{n(n-1)}{2}}.\label{eq:criterion3}
    \end{align}
    Then $P_{K,\ell,\chi_1,...,\chi_n}$ is nonsingular.
\end{proposition}
\begin{proof}
Note that $\det P_{K,\ell,\chi_1,...,\chi_n}=$
    \begin{align*}
        &\sum_{\tau\in S_n}(-1)^{\sign(\tau)}\prod_{i=1}^n\sigma_{\ell-1,\chi_i}(2^{\tau(i)-1})(1+\eta_{K,i,\tau(i)})\\=&\det \mathbf{P}_{K,\ell,\chi_1,...,\chi_n}+\sum_{\tau\in S_n}(-1)^{\sign(\tau)}f(\eta_{K,1,\tau(1)},...,\eta_{K,n,\tau(n)})\prod_{i=1}^n\sigma_{\ell-1,\chi_i}(2^{\tau(i)-1}).
    \end{align*}
  We have from Lemma \ref{lem:lowerboundoflimitmatrix3} that  
\begin{align*}
\left|\det \mathbf{P}_{K,\ell,\chi_1,...,\chi_n}\right|\geq\left(\frac{2^{\ell+1}}{D}\right)^{\frac{n(n-1)}{2}}.
   \end{align*}
 Note  also  that for any $\tau\in S_n$ we have 
\begin{align*}
    \left|\prod_{i=1}^n\sigma_{\ell-1,\chi}(2^{\tau(i)-1}) \right|&\leq\prod_{i=1}^n\sigma_{\ell-1}(2^{\tau(i)-1})\\&=\prod_{i=1}^n\sigma_{\ell-1}(2^{i-1})\\& < \prod_{i=2}^n2\cdot2^{(i-1)(\ell-1)}\\&=2^{n-1}2^{\frac{(\ell-1)n(n-1)}{2}},
\end{align*}
 which implies that
   \begin{align*}
       &\left|\sum_{\tau\in S_n}(-1)^{\sign(\tau)}f(\eta_{K,1,\tau(1)},...,\eta_{K,n,\tau(n)})\prod_{i=1}^n\sigma_{\ell-1,\overline{\chi}}(2^{\tau(i)-1})\right|\\\leq & n! \sup_{\eta\in S_n}|f(\eta_{K,1,\tau(1)},...,\eta_{K,n,\tau(n)})|2^{n-1}2^{\frac{(\ell-1)n(n-1)}{2}}\\ <&n!\cdot n!^{-1}\cdot2^{1-n}\cdot\left(\frac{4}{D}\right)^{\frac{n(n-1)}{2}}\cdot 2^{n-1}2^{\frac{(\ell-1)n(n-1)}{2}}\\<&\left(\frac{2^{\ell+1}}{D}\right)^{\frac{n(n-1)}{2}}.
   \end{align*}
   Therefore,  $\det P_{K,\ell,\chi_1,...,\chi_n}\neq 0$.
\end{proof}
Theorem \ref{thm:twistbydifferentD}  follows from the next result.
\begin{proposition}\label{prop:nonsingulartypeiii}
    Let $D\geq1$ be an odd square-free integer. For $K\gg_D1$, if $3\leq\ell\leq\frac{K-2}{2}$ is an integer, $\chi_1,...,\chi_n$ are primitive Dirichlet characters mod $D$ such that $\chi_i(-1)=(-1)^{\ell}$ and $\chi_i(2)$ are pairwise distinct for $1\leq i\leq n$, then $P_{K,\ell,\chi_1,...,\chi_n}$ is non-singular.
\end{proposition}
\begin{proof}
    Since $n$ is the number of primitive Dirichlet characters mod $D$ that are chosen, we have $n=O(D)$.   Again, by Proposition \ref{prop:asymptoticA}, we have $\eta_{K,i,\tau(i)}=o(1)$ as for all $1\leq i\leq n$ and $\tau\in S_n$, where $o(1)\rightarrow0$ when $K$ is sufficiently large relative to $D$. Then Lemma \ref{lem:inequalityforf} implies that
    \begin{align*}
          \sup_{\tau\in S_n}|f(\eta_{K,1,\tau(1)},...,\eta_{K,n,\tau(n)})|&\leq(1+o(1))^n-1=o(1).
    \end{align*}
    In particular, $K$ can be chosen such that \eqref{eq:criterion3} is satisfied, which shows that $P_{K,\ell,\chi_1,...,\chi_n}$ is non-singular as $K\gg_{D}1$.
\end{proof}

Finally, we study the matrix that is used to prove Theorem \ref{thm:twistbydifferentD2}.
Let $\chi_1,...,\chi_n$ be primitive Dirichlet characters mod $D$ such that $\chi_i(-1)=(-1)^{\ell}$  and $\chi_i(2)$ are pairwise distinct for all $1\leq i\leq n$. 
Define the coefficient matrix
\begin{align}
    Q_{K,\ell,\chi_1,...,\chi_n}:=\left[\mathfrak{b}_{K,\ell,\chi_i}(2^{j-1})\right]_{1\leq i,j\leq n}=\left[2^{j-1}\sigma_{\ell-1,\chi_i}(2^{j-1})(1+\nu_{K,i,j})\right]_{1\leq i,j\leq n}.\label{eq:matrix4}
\end{align}
and the matrix
\begin{align}
\mathbf{Q}_{K,\ell,\chi_1,...,\chi_n}:&=\begin{bmatrix}
        1 & 2\sigma_{\ell-1,\chi_1}(2) & 2^2\sigma_{\ell-1,\chi_1}(2^2) & \cdots & 2^{n-1}\sigma_{\ell-1,\chi_1}(2^{n-1})\\
        1 &  2\sigma_{\ell-1,\chi_2}(2) &  2^2\sigma_{\ell-1,\chi_2}(2^2) & \cdots &  2^{n-1}\sigma_{\ell-1,\chi_2}(2^{n-1})\\
        \vdots & \vdots & \vdots & \ddots & \vdots\\
        1 & 2\sigma_{\ell-1,\chi_{n-1}}(2) & 2^2\sigma_{\ell-1,\chi_{n-1}}(2^2) & \cdots & 2^{n-1}\sigma_{\ell-1,\chi_{n-1}}(2^{n-1})\\ 
        1 & 2\sigma_{\ell-1,\chi_n}(2) & 2^2\sigma_{\ell-1,\chi_n}(2^2) & \cdots & 2^{n-1}\sigma_{\ell-1,\chi_n}(2^{n-1})
    \end{bmatrix}.\label{eq:limitmatrixtype4}
\end{align}
\begin{lemma}\label{lem:lowerboundoflimitmatrix4}
   Let $\mathbf{Q}_{K,\ell,\chi_1,...,\chi_n}$ be defined as in \eqref{eq:limitmatrixtype4}. Then 
   \begin{align*}
\left|\det \mathbf{Q}_{K,\ell,\chi_1,...,\chi_n}\right|\geq\left(\frac{2^{\ell+2}}{D}\right)^{\frac{n(n-1)}{2}}.
   \end{align*}
\end{lemma}
\begin{proof}
  Since  $\det \mathbf{Q}_{K,\ell,\chi_1,...,\chi_n}=2^{\frac{n(n-1)}{2}}\det\mathbf{P}_{K,\ell,\chi_1,...,\chi_n}$, the result follows from Lemma \ref{lem:lowerboundoflimitmatrix3}.
\end{proof}
\begin{proposition}
    Let $f$ be defined as in \eqref{eq:functionf}.  Suppose that 
    \begin{align*}
        \sup_{\tau\in S_n}|f(\nu_{K,1,\tau(1)},...,\nu_{K,n,\tau(n)})|<n!^{-1}\cdot2^{1-n}\cdot\left(\frac{4}{D}\right)^{\frac{n(n-1)}{2}}.
    \end{align*}
    Then $Q_{K,\ell,\chi_1,...,\chi_n}$ is nonsingular.
\end{proposition}
\begin{proof}
Note that $\det Q_{K,\ell,\chi_1,...,\chi_n}=$
    \begin{align*}
        &\sum_{\tau\in S_n}(-1)^{\sign(\tau)}\prod_{i=1}^n2^{\tau(i)-1}\sigma_{\ell-1,\chi_i}(2^{\tau(i)-1})(1+\nu_{K,i,\tau(i)})\\=&\det \mathbf{Q}_{K,\ell,\chi_1,...,\chi_n}+\sum_{\tau\in S_n}(-1)^{\sign(\tau)}f(\nu_{K,1,\tau(1)},...,\nu_{K,n,\tau(n)})\prod_{i=1}^n2^{\tau(i)-1}\sigma_{\ell-1,\chi_i}(2^{\tau(i)-1}).
    \end{align*}
By  Lemma \ref{lem:lowerboundoflimitmatrix4},  we have  
\begin{align*}
\left|\det \mathbf{Q}_{K,\ell,\chi_1,...,\chi_n}\right|\geq\left(\frac{2^{\ell+2}}{D}\right)^{\frac{n(n-1)}{2}}.
   \end{align*}
For any $\tau\in S_n$ we have 
\begin{align*}   \left|\prod_{i=1}^n2^{\tau(i)-1}\sigma_{\ell-1,\chi}(2^{\tau(i)-1}) \right|&\leq\prod_{i=1}^n2^{\tau(i)-1}\sigma_{\ell-1}(2^{\tau(i)-1})\\&=\prod_{i=1}^n2^{i-1}\sigma_{\ell-1}(2^{i-1})\\& < \prod_{i=2}^n2^{i-1}\cdot 2\cdot2^{(i-1)(\ell-1)}\\&=2^{n-1}2^{\frac{\ell n(n-1)}{2}}.
\end{align*}
Hence
   \begin{align*}
       &\left|\sum_{\tau\in S_n}(-1)^{\sign(\tau)}f(\nu_{K,1,\tau(1)},...,\nu_{K,n,\tau(n)})\prod_{i=1}^n2^{\tau(i)-1}\sigma_{\ell-1,\overline{\chi}}(2^{\tau(i)-1})\right|\\\leq & n! \sup_{\tau\in S_n}|f(\nu_{K,1,\tau(1)},...,\nu_{K,n,\tau(n)})|2^{n-1}2^{\frac{\ell n(n-1)}{2}}\\ <&n!\cdot n!^{-1}\cdot2^{1-n}\cdot\left(\frac{4}{D}\right)^{\frac{n(n-1)}{2}}\cdot 2^{n-1}2^{\frac{\ell n(n-1)}{2}}\\<&\left(\frac{2^{\ell+2}}{D}\right)^{\frac{n(n-1)}{2}}.
   \end{align*}
   Therefore,  $\det Q_{K,\ell,\chi_1,...,\chi_n}\neq 0$.
\end{proof}
The next proposition can be proved the same way as the proof of Proposition \ref{prop:nonsingulartypeiii}. 
\begin{proposition}\label{prop:twisttypeiv}
    Let $D\geq1$ be an odd square-free integer. For $K\gg_D1$, if $3\leq\ell\leq\frac{K-4}{2}$ is an integer, $\chi_1,...,\chi_n$ are primitive Dirichlet characters mod $D$ such that $\chi_i(-1)=(-1)^{\ell}$ and $\chi_i(2)$ are pairwise distinct for $1\leq i\leq n$, then $Q_{K,\ell,\chi_1,...,\chi_n}$ is non-singular.
\end{proposition}
Thus, we have completed the proof of Theorem \ref{thm:twistbydifferentD2}.
\begin{remark}
    In fact, we could also consider the $q^{p^i}$-Fourier coefficients ($p$ is any arbitrary fixed prime) for matrices $P_{K,\ell,\chi_1,...,\chi_n}$ and $Q_{K,\ell,\chi_1,...,\chi_n}$, and the assumption that $\chi_i$'s take distinct values at $2$ in Propositions \ref{prop:nonsingulartypeiii} and \ref{prop:twisttypeiv} can be replaced by $\chi_i$'s taking distinct values at an arbitrary fixed prime $p$, where the results will hold for $K\gg_{p,D}1$.
\end{remark}
\section{Applications}\label{sect:applications}

\subsection{Convolution sums of twisted divisor functions}
Note that $a_{K,\ell,\chi}(n)= b_{K,\ell,\chi}(n)=0$ for all $n\ge1$ when $K<12$ or $K=14$ because $\dim S_K=0$. This fact, together with Propositions \ref{prop:Fourco} and \ref{prop:FourcoRankinCohen}, will give identities that evaluate convolution sums of twisted divisor functions by generalized Bernoulli numbers.  

For a Dirichlet character $\chi$ mod $D$, the generalized Bernoulli numbers $B_{n,\chi}$ are defined by
\begin{align*}
    \sum_{a=1}^D\frac{\chi(a)te^{at}}{e^{Dt}-1}=\sum_{n=0}^{\infty}B_{n,\chi}\frac{t^n}{n!}.
\end{align*}
If $\chi$ is primitive mod $D$ such that $\chi(-1)=(-1)^k$, then the special value of $L(s,\chi)$  is related to $B_{k,\chi}$ by (\cite[Theorems 9.6]{bernoullinumberbook})
\begin{align}
    L(k,\chi)=\frac{(-1)^{k-1}(2\pi i)^k}{2k!D^k}G(\chi)B_{k,\overline{\chi}}.\label{eq:Lseriesandbernoulli}
\end{align}

To obtain clean formulas,  we restrict to the coefficients $a_{K,\ell,\chi}(1)$ and $b_{K,\ell,\chi}(1)$ for $D=p$. Readers may compare the following results with the classical ones, see for example \cite{convolutionsum}.
\begin{theorem}\label{thm:a=0}
    Let $p>2$ be a prime number and  $\chi$ be a primitive Dirichlet character mod $p$. 
    \begin{enumerate} 
    \item If $\chi(-1)=-1$, then we have 
    \begin{align*}
       \sum_{n=1}^{p}\sigma_{2,\mathbbm{1},\overline{\chi}}(n)\sigma_{2,\mathbbm{1},\chi}(p-n)&= -\frac{B_{3,\overline{\chi}}}{6}-\frac{B_{3,\chi}}{6}+\frac{B_{3,\overline{\chi}}B_{3,\chi}}{3B_6},\\\sum_{n=1}^{p}\sigma_{2,\mathbbm{1},\overline{\chi}}(n)\sigma_{4,\mathbbm{1},\chi}(p-n)&=-\frac{B_{5,\overline{\chi}}}{10}-\frac{B_{3,\chi}}{6}+\frac{4B_{5,\overline{\chi}}B_{3,\chi}}{15B_8},\\\sum_{n=1}^{p}\sigma_{2,\mathbbm{1},\overline{\chi}}(n)\sigma_{6,\mathbbm{1},\chi}(p-n)&=-\frac{B_{7,\overline{\chi}}}{14}-\frac{B_{3,\chi}}{6}+\frac{5B_{7,\overline{\chi}}B_{3,\chi}}{21B_{10}},\\\sum_{n=1}^{p}\sigma_{2,\mathbbm{1},\overline{\chi}}(n)\sigma_{10,\mathbbm{1},\chi}(p-n)&=-\frac{B_{11,\overline{\chi}}}{22}-\frac{B_{3,\chi}}{6}+\frac{7B_{11,\overline{\chi}}B_{3,\chi}}{33B_{14}},\\\sum_{n=1}^{p}\sigma_{4,\mathbbm{1},\overline{\chi}}(n)\sigma_{4,\mathbbm{1},\chi}(p-n)&=-\frac{B_{5,\overline{\chi}}}{10}-\frac{B_{5,\chi}}{10}+\frac{B_{5,\overline{\chi}}B_{5,\chi}}{5B_{10}},\\\sum_{n=1}^{p}\sigma_{4,\mathbbm{1},\overline{\chi}}(n)\sigma_{8,\mathbbm{1},\chi}(p-n)&=-\frac{B_{9,\overline{\chi}}}{18}-\frac{B_{5,\chi}}{10}+\frac{7B_{9,\overline{\chi}}B_{5,\chi}}{45B_{14}},\\\sum_{n=1}^{p}\sigma_{6,\mathbbm{1},\overline{\chi}}(n)\sigma_{6,\mathbbm{1},\chi}(p-n)&=-\frac{B_{7,\overline{\chi}}}{14}-\frac{B_{7,\chi}}{14}+\frac{7B_{7,\overline{\chi}}B_{7,\chi}}{7B_{14}}.
    \end{align*}
    \item If $\chi(-1)=1$, then we have
    \begin{align*}
        \sum_{n=1}^{p}\sigma_{3,\mathbbm{1},\overline{\chi}}(n)\sigma_{3,\mathbbm{1},\chi}(p-n)&=\frac{B_{4,\overline{\chi}}}{8}+\frac{B_{4,\chi}}{8}-\frac{B_{4,\overline{\chi}}B_{4,\chi}}{4B_{8}},\\    \sum_{n=1}^{p}\sigma_{3,\mathbbm{1},\overline{\chi}}(n)\sigma_{5,\mathbbm{1},\chi}(p-n)&=\frac{B_{6,\overline{\chi}}}{12}+\frac{B_{4,\chi}}{8}-\frac{5B_{6,\overline{\chi}}B_{4,\chi}}{24B_{10}},\\     \sum_{n=1}^{p}\sigma_{3,\mathbbm{1},\overline{\chi}}(n)\sigma_{9,\mathbbm{1},\chi}(p-n)&=\frac{B_{10,\overline{\chi}}}{20}+\frac{B_{4,\chi}}{8}-\frac{7B_{10,\overline{\chi}}B_{4,\chi}}{40B_{14}}.
    \end{align*}
    \end{enumerate}
\end{theorem}
\begin{proof}
Note that \eqref{eq:twistsumdivisor} and \eqref{eq:Lseriesandbernoulli} imply that
\begin{align}
    \sigma_{k-1,\chi}(0)=-\frac{B_{k,\chi}}{2k}.\label{eq:twistsigma(0_}
\end{align}
Using Proposition \ref{prop:Fourco} for $D=p$, and \eqref{eq:twistsigma(0_}, we get
    \begin{align*}
      a_{K,\ell,\chi}(1)=&\sigma_{k-1,\overline{\chi}}(0)+\sigma_{\ell-1,\chi}(0)+\overline{\chi}(-1)\sum_{n=1}^{p-1}\sigma_{\ell-1,\mathbbm{1},\overline{\chi}}(n)\sigma_{k-1,\mathbbm{1},\chi}(p-n)-\frac{2\sigma_{\ell-1,\chi}(0)\sigma_{k-1,\overline{\chi}}(0)}{\zeta(1-K)}\\=&-\frac{B_{k,\overline{\chi}}}{2k}-\frac{B_{\ell,\chi}}{2\ell}-\frac{2\left(-\frac{B_{\ell,\chi}}{2\ell}\right)\left(-\frac{B_{k,\overline{\chi}}}{2k}\right)}{-\frac{B_K}{K}}+\overline{\chi}(-1)\sum_{n=1}^{p-1}\sigma_{\ell-1,\mathbbm{1},\overline{\chi}}(n)\sigma_{k-1,\mathbbm{1},\chi}(p-n)\\=&-\frac{B_{k,\overline{\chi}}}{2k}-\frac{B_{\ell,\chi}}{2\ell}+\frac{KB_{k,\overline{\chi}}B_{\ell,\chi}}{2k\ell  B_{K}}+\overline{\chi}(-1)\sum_{n=1}^{p-1}\sigma_{\ell-1,\mathbbm{1},\overline{\chi}}(n)\sigma_{k-1,\mathbbm{1},\chi}(p-n),
    \end{align*}
 where $k=K-\ell$. Note that if 
 $$(\ell,k)\in \mathcal{S}:=\{(3,3),(3,5),(3,7),(3,11);(4,4),(4,6),(4,10);(5,5),(5,9);(7,7) \},$$ then $0\equiv
 \mathcal{F}_{K,\ell,\chi}\in S_{K}$ since for those values of $(\ell,k)$ we have $\dim S_K=0.$
 It follow that for all $(\ell,k)\in\mathcal{S}$ we have
    \begin{align}
        -\frac{B_{k,\overline{\chi}}}{2k}-\frac{B_{\ell,\chi}}{2\ell}+\frac{KB_{k,\overline{\chi}}B_{\ell,\chi}}{2k\ell B_{K}}+\overline{\chi}(-1)\sum_{n=1}^{p-1}\sigma_{\ell-1,\mathbbm{1},\overline{\chi}}(n)\sigma_{k-1,\mathbbm{1},\chi}(p-n)=0.\label{eq:identityOne}
    \end{align}
Writing out \eqref{eq:identityOne} explicitly gives the result.
\end{proof}
\begin{theorem}\label{thm:b=0}
    Let $p>2$ be a prime number and  $\chi$ be a primitive Dirichlet character mod $p$.
    \begin{enumerate} 
    \item 
    If $\chi(-1)=-1$, then
    \begin{align*}
        \sum_{n=1}^{p-1}\sigma_{2,\mathbbm{1},\overline{\chi}}(n)\sigma_{4,\mathbbm{1},\chi}(p-n)(3p-8n)&=\frac{p}{2}(B_{5,\overline{\chi}}-B_{3,\chi}),\\ \sum_{n=1}^{p-1}\sigma_{4,\mathbbm{1},\overline{\chi}}(n)\sigma_{6,\mathbbm{1},\chi}(p-n)(5p-12n)&=\frac{p}{2}(B_{7,\overline{\chi}}-B_{5,\chi}),\\\sum_{n=1}^{p-1}\sigma_{2,\mathbbm{1},\overline{\chi}}(n)\sigma_{8,\mathbbm{1},\chi}(p-n)(p-4n)&=\frac{p}{6}(B_{9,\overline{\chi}}-B_{3,\chi}).
 \end{align*}
 \item  If $\chi(-1)=1$, then
 \begin{align}
 \sum_{n=1}^{p-1}\sigma_{3,\mathbbm{1},\overline{\chi}}(n)\sigma_{7,\mathbbm{1},\chi}(p-n)(p-3n)=-\frac{p}{8}(B_{8,\overline{\chi}}-B_{4,\chi}).  \end{align}
 \end{enumerate}
\end{theorem}
\begin{proof}
    Applying Proposition \ref{prop:FourcoRankinCohen} for $D=p$, we get 
    \begin{align*}
        b_{K,\ell,\chi}(1)&=-k\sigma_{k-1,\overline{\chi}}(0)+\ell \sigma_{\ell-1,\chi}(0)+\overline{\chi}(-1)p^{-1}\sum_{\substack{a_1,a_2\geq0\\a_1+a_2=p}}\sigma_{\ell-1,\mathbbm{1},\overline{\chi}}(a_1)\sigma_{k-1,\mathbbm{1},\chi}(a_2)(\ell a_2-k a_1)\\&=\frac{B_{k,\overline{\chi}}}{2}-\frac{B_{\ell,\chi}}{2}+\overline{\chi}(-1)p^{-1}\sum_{n=1}^{p-1}\sigma_{\ell-1,\mathbbm{1},\overline{\chi}}(n)\sigma_{k-1,\mathbbm{1},\chi}(p-n)(\ell (p-n)-kn),
    \end{align*}
where $k=K-2-\ell$. Note that if $(\ell,k)\in \mathcal{T}:=\{(3,5),(3,9),(4,8),(5,7)\}$ then $0\equiv \mathcal{G}_{K,\ell,\chi}\in S_K$ since for those values of $(\ell,k)$ we have $\dim S_K=0$.
It follows that for all $(\ell,k)\in\mathcal{T}$  we have 
    \begin{align*}
        \frac{B_{k,\overline{\chi}}}{2}-\frac{B_{\ell,\chi}}{2}+\overline{\chi}(-1)p^{-1}\sum_{n=1}^{p-1}\sigma_{\ell-1,\mathbbm{1},\overline{\chi}}(n)\sigma_{k-1,\mathbbm{1},\chi}(p-n)(\ell (p-n)-kn)=0,
    \end{align*}
    which gives the result.
\end{proof}
\begin{remark}
  Some small cases  ($D=1,3,5,7,11,13,15$) of Theorems \ref{thm:a=0} and \ref{thm:b=0}  have been numerically verified by \cite{codefortwisted}, see also Section \ref{sect:conj}.  
\end{remark}

\subsection{Non-vanishing of twisted central \texorpdfstring{$L$-}-values}
We now focus on the case when $\chi$ is real, or equivalently is quadratic. As $G(\chi)^2=\chi(-1) D$ in this case, for a normalized Hecke eigenform $f\in S_{K}$ the functional equation \eqref{eq:functionaleqoftwistedL} becomes
\begin{align*}
    \Lambda(f,\chi,s)=(-1)^{K/2}\chi(-1)\Lambda(f,\chi,K-s).
\end{align*}
Note that if $(-1)^{K/2}\chi(-1)<0$ then $L(f,\chi,K/2)$ is trivially zero. 
Some evidence has been provided for the non-vanishing of $L(f,\chi,K/2)$ when $(-1)^{K/2}\chi(-1)>0$.
For example, Lau and Tsang \cite{LautwistedcentralLvalues} showed that for $1\leq D\leq K^{1/6}/(\log K)^5$, we have
\begin{align*}
    \#\{f\in\mathcal{H}_K:L(f,\chi,K/2)\neq0\}\gg|1+\epsilon_{K}(\chi)|^2\frac{D}{\varphi(D)}\frac{K}{(\log K)^2},
\end{align*}
where $\epsilon_K(\chi)=i^KG(\chi)^2/D$, and $\mathcal{H}_K$ denotes the set of normalized Hecke eigenform in $S_{K}$.

As an application of our methods, we show that Maeda's conjecture implies that $L(f,\chi,K/2)$ is non-vanishing as $K\geq10D+2$. Following \cite{Maeda'sconjecture}, we first review  Maeda's conjecture. Let $f(z)=\sum_{n=1}^{\infty}a_f(n)q^n\in S_K$ be a normalized Hecke eigenform. Denote by $K_f$  the field generated over $\mathbb{Q}$ by the coefficients $a_f(n)$ for all $n$. Let $G(f)$ be the Galois group of the Galois closure of $K_f$ over $\mathbb{Q}$. It is well-known that $K_f$ is a number field, and for every $\sigma\in G(f)$, 
\begin{align*}
    f^{\sigma}(z):=\sum_{n=1}^{\infty}a_{f}(n)^{\sigma}q^n
\end{align*}
is also a normalized Hecke eigenform in $S_K$. Maeda's conjecture can be formulated as follows.
\begin{conjecture}\cite[p.~306]{Maeda'sconjecture}
    Let $f\in S_K$ be a normalized Hecke eigenform. Then
    \begin{itemize}
        \item[($H_a$)] $\Span\{f^{\sigma}:\sigma\in G(f)\}=S_K$;
        \item[($H_b$)] $G_f$ is isomorphic to the symmetric group of degree $\dim S_K$.
    \end{itemize}
\end{conjecture}
\begin{proposition}\label{prop:reductionbyshimuraresult}
  Let $\mathcal{H}_K$ denote the set of normalized Hecke eigenform in $S_K$. Assume $(H_a)$ holds true for $S_K$.  Suppose there exists some $f\in \mathcal{H}_K$ such that $L(f,\chi,K/2)\neq0$. Then $L(g,\chi,K/2)\neq0$ for all $g\in\mathcal{H}_K$. 
\end{proposition}
\begin{proof}
    Since $L(f,\chi,K/2)\neq0$, we have $L(f^{\sigma},\chi^{\sigma},K/2)=L(f^{\sigma},\chi,K/2)\neq0$ for all $\sigma\in G(f)$ by
     Shimura's result \cite[Theorem 1]{Shimura1976}.  For any $g\in\mathcal{H}_K$, there exists some $\tau\in G(f)$ such that $g=f^{\tau}$ by assumption. Hence $L(g,\chi,K/2)=L(f^{\tau},\chi,K/2)\neq0$.
\end{proof}
The following result can be viewed as a generalization of \cite[Theorem 1]{ConreyFarmernonvanshing}.
\begin{theorem}
    Let $D\geq1$ be an odd square-free integer and $\chi$ be a primitive quadratic Dirichlet character $D$ such that $(-1)^{K/2}\chi(-1)>0$. If $(H_a)$ holds for $S_K$ and $K\geq 10D+2$, then for every Hecke eigenform $f\in S_K$
    \begin{align*}
        L(f,\chi,K/2)\neq 0.
    \end{align*}
\end{theorem}
\begin{proof}
    As mentioned in Remark \ref{remark:centralprod}, \eqref{eq:convolution} extends to $k=\ell=\frac{K}{2}$:
    \begin{align*}
        \langle f,\mathcal{F}_{K,k,\chi}\rangle=\frac{\Gamma(K-1)(k-1)!D^{k}}{(4\pi)^{K-1}(2\pi i)^{k}\overline{G(\chi)}}L(f,K-1)L(f,\chi,K/2).
    \end{align*}
   Since $L(f,K-1)\neq0$, we have
    \begin{align*}
        L(f,\chi,K/2)=0\quad{\rm if~and~only~if}\quad\langle f,\mathcal{F}_{K,k,\chi}\rangle=0.
    \end{align*}
    By Proposition \ref{prop:reductionbyshimuraresult}, it suffices to show that $\mathcal{F}_{K,k,\chi}\not\equiv0$ for $K\geq10D+2$, which is proved in Proposition \ref{prop:nonvanishingofF}.
\end{proof}
 To show that $\mathcal{F}_{K,k,\chi}\not\equiv0$, it suffices to prove that at least one Fourier coefficient of $\mathcal{F}_{K,k,\chi}$  is non-vanishing. It is convenient to consider the $q$-Fourier coefficient $a_{K,k,\chi}(1)$ of $\mathcal{F}_{K,k,\chi}$. 

\begin{proposition}\label{prop:nonvanishingofF}  Let $D\geq1$ be an odd square-free integer and $\chi$ be a primitive quadratic Dirichlet character $D$ such that $(-1)^{K/2}\chi(-1)>0$.
 If $K\geq10D+2$, then $a_{K,k,\chi}(1)\neq0$. 
\end{proposition}
\begin{proof}
By Proposition \ref{prop:Fourco}, Corollary \ref{co:cuspproj} (and Remark \ref{remark:traceofprodwhenl=k}), and the fact that $\chi=\overline{\chi}$, we have
\begin{align*}
a_{K,k,\chi}(1)
=2\sigma_{k-1,\chi}(0)+\sum_{\substack{D=D_1D_2\\D_2\neq 1 }}\chi_2(-1)^{-1}\sum_{\substack{a_1,a_2\geq0\\a_1+a_2=D_2}}\sigma_{k-1,\chi_1,\chi_2}(a_1)\sigma_{k-1,\chi_1,\chi_2}(a_2)-\frac{2\sigma_{k-1,\chi}(0)^2}{\zeta(1-K)}.
\end{align*}
Thus, we can write $a_{K,k,\chi}(1)=2\sigma_{k-1,\chi}(0)(1+\varepsilon_{K,D})$, where
\begin{align*}
    \varepsilon_{K,D}=\frac{\sigma_{k-1,\chi}(0)^{-1}}{2}\sum_{\substack{D=D_1D_2\\D_2\neq 1 }}\chi_2(-1)^{-1}\sum_{\substack{a_1,a_2\geq0\\a_1+a_2=D_2}}\sigma_{k-1,\chi_1,\chi_2}(a_1)\sigma_{\ell-1,\chi_1,\chi_2}(a_2)-\frac{\sigma_{k-1,\chi}(0)}{\zeta(1-K)}.
\end{align*}
To show $a_{K,k,\chi}(1)\neq0$, it suffices to show that $|\varepsilon_{K,D}|<1$. 
Note that \eqref{eq:accuratebdquotientzetaandL} implies that
\begin{align}
    \left|\frac{\sigma_{k-1,\chi}(0)}{\zeta(1-K)}\right|&\leq   \frac{e\zeta(k)}{2\sqrt{2\pi}}\left(\frac{(k-1)D}{K-1}\right)^{k-1/2}\left(\frac{2\pi e}{K-1}\right)^{k}  \nonumber\\&\leq\frac{e\zeta(k)}{2\sqrt{2\pi}}\sqrt{\frac{K-1}{(k-1)D}}\left(\frac{(K-2)\pi eD}{(K-1)^2}\right)^k  \nonumber\\&\leq\frac{e\zeta(k)}{2\sqrt{2\pi}}\sqrt{\frac{3}{D}}\left(\frac{\pi e D}{K-1}\right)^k.  \label{eq:bdfor2ndterm}
\end{align}
Using \eqref{eq:accuratebdforcentralL} for $\ell=k$ and $n=1$, we get
 \begin{align}
\left|\sum_{\substack{D=D_1D_2\\D_2\neq 1 }}\overline{\chi}_2(-1)\sum_{\substack{a_1,a_2\geq0\\a_1+a_2=D_2}}\sigma_{k-1,\chi_1,\chi_2}(a_1)\sigma_{k-1,\chi_1,\chi_2}(a_2)\right|\leq2\zeta(k-1)^2\zeta(K-1)\left(\frac{D}{2}\right)^{K-1}. \label{eq:sharperbdmathfrakE}
 \end{align}
Note also that \eqref{eq:sharperbd2/L} implies that 
    \begin{align}
        \left|\frac{\sigma_{k-1,\overline{\chi}}(0)^{-1}}{2}\right|\leq\frac{1}{2\sqrt{e}(2-\zeta(k))}\left(\frac{2\pi e}{(k-1)D}\right)^{k-1/2}.\label{eq:bd1/Lsharper}
    \end{align}
Combing \eqref{eq:bdfor2ndterm}, \eqref{eq:sharperbdmathfrakE} and \eqref{eq:bd1/Lsharper}, we get 
\begin{align*}
   |\varepsilon_{K,D}|\leq  &\frac{\zeta(k-1)^2\zeta(K-1)}{\sqrt{e}(2-\zeta(k))}\left(\frac{\pi eD}{K-2}\right)^{\frac{K-1}{2}}+   \frac{e\zeta(k)}{2\sqrt{2\pi}}\sqrt{\frac{3}{D}}\left(\frac{\pi e D}{K-1}\right)^k.
\end{align*}
If $K-2\geq 10D$, then  
\begin{align*}
    |\varepsilon_{K,D}|<\frac{\zeta(5)^2\zeta(11)}{\sqrt{e}(2-\zeta(6))}\left(\frac{\pi e}{10}\right)^{\frac{11}{2}}+   \frac{e\zeta(6)}{2\sqrt{2\pi}}\sqrt{3}\left(\frac{\pi e }{10}\right)^6\approx0.65<1,
\end{align*}
which completes the proof.
\end{proof}

\section{Conjectures}\label{sect:conj}
We believe that twisted periods with fixed character but different indices span the whole vector space $S_K^{\ast}$, which follows from the conjectures below. 
\begin{conjecture}\label{conj:differentindice}
 Let $3\leq\ell_1<\ell_2<...<\ell_n\leq \frac{K-2}{2}$ be integers,
 $D\geq1$ an integer, and $\chi$ be a primitive Dirichlet character such that $\chi(-1)=(-1)^{\ell_i}$ for all $1\leq i\leq n$.  If $n\leq\dim S_K$, then $\{r_{\ell_i-1,\chi}\}_{i=1}^{n}$  is linearly independent.
\end{conjecture}
\begin{conjecture}\label{conj:differentindicebracket}
  Let $3\leq\ell_1<\ell_2<...<\ell_n\leq\frac{K-4}{2}$ be integers, $D\geq1$ an integer, and $\chi$ be a primitive Dirichlet character such that $\chi(-1)=(-1)^{\ell_i}$ for all $1\leq i\leq n$.  If $n\leq\dim S_K$, then $\{r_{\ell_i,\chi}\}_{i=1}^{n}$  is linearly independent.  
\end{conjecture}

The above conjectures are analogues of \cite[Conjecture 1.5]{oddperiods} and \cite[Conjecture 4.5]{evenperiods}, respectively. 
We verified Conjectures \ref{conj:differentindice} and \ref{conj:differentindicebracket} for $K,D\leq 40$ \cite{codefortwisted} by  checking the non-singularity of the matrices
\begin{align*}
   \mathcal{C}_1:=& \begin{bmatrix}
        a_{K,\ell_1,\chi}(1) & a_{K,\ell_1,\chi}(2) & \cdots & a_{K,\ell_1,\chi}(n) \\ a_{K,\ell_2,\chi}(1) & a_{K,\ell_2,\chi}(2) & \cdots & a_{K,\ell_n,\chi}(n) \\ \vdots &\vdots &\ddots &\vdots\\a_{K,\ell_n,\chi}(1) & a_{K,\ell_n,\chi}(2) & \cdots & a_{K,\ell_n,\chi}(n) 
    \end{bmatrix}\end{align*}
and    
    \begin{align*}
    \mathcal{C}_2:=&\begin{bmatrix}
        b_{K,\ell_1,\chi}(1) & b_{K,\ell_1,\chi}(2) & \cdots & b_{K,\ell_1,\chi}(n) \\ b_{K,\ell_2,\chi}(1) & b_{K,\ell_2,\chi}(2) & \cdots & b_{K,\ell_n,\chi}(n) \\ \vdots &\vdots &\ddots &\vdots\\b_{K,\ell_n,\chi}(1) & b_{K,\ell_n,\chi}(2) & \cdots & b_{K,\ell_n,\chi}(n) 
    \end{bmatrix},
\end{align*}
where $a_{K,\ell,\chi}(n)$ and $b_{K,\ell,\chi}(n)$ denote the $q^n$-Fourier coefficient of $\mathcal{F}_{K,\ell,\chi}(z)$ and $\mathcal{G}_{K,\ell,\chi}(z)$, respectively.

Let $N(D)$ denote the number of primitive Dirichlet characters mod $D$. We also believe that twisted periods with fixed index but different twists mod $D$ span $S_K^{\ast}$ as long as $N(D)\geq \dim S_K$, which are corollaries of the following conjectures. 
 \begin{conjecture}\label{conj:differentchi}
       Let $D\geq1$ and $3\leq\ell\leq\frac{K-2}{2}$  be integers 
       , and $\chi_1,...,\chi_n$ be distinct primitive Dirichlet characters mod $D$ such that $\chi_i(-1)=(-1)^{\ell}$. If $n\leq \dim S_K$, then the set of twisted periods $\{r_{\ell-1,\chi_i}\}_{i=1}^n$ on $S_K$ is linearly independent.
 \end{conjecture}   
 \begin{conjecture}\label{conj:differentchibracket}
       Let $D\geq1$ and $3\leq\ell\leq\frac{K-4}{2}$  be integers, and $\chi_1,...,\chi_n$ be distinct primitive Dirichlet characters mod $D$ such that $\chi_i(-1)=(-1)^{\ell}$. If $n\leq \dim S_K$, then the set of twisted periods $\{r_{\ell,\chi_i}\}_{i=1}^n$ on $S_K$ is linearly independent.
 \end{conjecture}   
We also verified Conjectures \ref{conj:differentchi} and \ref{conj:differentchibracket} for $K,D\leq40$ \cite{codefortwisted} by checking  matrices 
\begin{align*}
  \mathcal{C}_3:=\begin{bmatrix}
        a_{K,\ell,\chi_1}(1) & a_{K,\ell,\chi_1}(2) & \cdots & a_{K,\ell,\chi_1}(n) \\ a_{K,\ell,\chi_2}(1) & a_{K,\ell,\chi_2}(2) & \cdots & a_{K,\ell,\chi_2}(n) \\ \vdots &\vdots &\ddots &\vdots\\a_{K,\ell,\chi_n}(1) & a_{K,\ell,\chi_n}(2) & \cdots & a_{K,\ell,\chi_n}(n) 
    \end{bmatrix}\end{align*}
    and
    \begin{align*}
    \mathcal{C}_4:=\begin{bmatrix}
        b_{K,\ell,\chi_1}(1) & b_{K,\ell,\chi_1}(2) & \cdots & b_{K,\ell,\chi_1}(n) \\ b_{K,\ell,\chi_2}(1) & b_{K,\ell,\chi_2}(2) & \cdots & b_{K,\ell,\chi_2}(n) \\ \vdots &\vdots &\ddots &\vdots\\b_{K,\ell,\chi_n}(1) & b_{K,\ell,\chi_n}(2) & \cdots & b_{K,\ell,\chi_n}(n) 
    \end{bmatrix}.
\end{align*}
are non-singular.

\section*{Acknowledgments}
The authors thank the anonymous referee for the detailed comments and insightful advice that have improved the exposition of this article. Research of Hui Xue is supported by  Simons Foundation Grant MPS-TSM-00007911. We also thank Erick Ross for numerically verifying Conjectures \ref{conj:differentindice}, \ref{conj:differentindicebracket}, \ref{conj:differentchi} and \ref{conj:differentchibracket}. 
\bibliographystyle{plain}
  
  \vspace{.2in}


\providecommand{\bysame}{\leavevmode\hbox
to3em{\hrulefill}\thinspace}

\bibliography{ref}
\end{document}